\documentclass[11 pt]{amsart}
\usepackage{amscd, amssymb,url, tikz, nicefrac}
\usepackage[all]{xy}

\usepackage[utf8]{inputenc}
\usepackage[T1]{fontenc}

\usepackage[margin=3cm]{geometry} 

\usetikzlibrary{calc}
\usepackage{array}
\usepackage[shortlabels]{enumitem}

\usepackage{hyperref}
\hypersetup{
 colorlinks = true,
 linkcolor = blue,
 filecolor = blue, 
 urlcolor = blue,
 citecolor = blue,
}

\newtheorem{thm}[subsubsection]{Theorem}
\newtheorem{pro}[subsubsection]{Proposition}
\newtheorem{lem}[subsubsection]{Lemma}
\newtheorem{cor}[subsubsection]{Corollary}

\theoremstyle{definition} 
	\newtheorem{defn}[subsubsection]{Definition}

\theoremstyle{remark} 
	\newtheorem{rem}[subsubsection]{Remark}
	\newtheorem{exa}[subsubsection]{Example}
	
		\newtheorem{assume}{Assumption}

\newcommand{\mrage}{\mathrm{age}}
\newcommand{\diag}{\mathrm{diag}}
\newcommand{\mf}{\mathfrak}

\newcommand{\CC}{\mathbb C}

\newcommand{\NN}{\mathbb N}
\newcommand{\PP}{\mathbb P}
\newcommand{\GG}{\mathbb G}
\newcommand{\QQ}{\mathbb Q}

\newcommand{\ZZ}{\mathbb Z}

\newcommand{\Ocal}{\mathcal O}

\newcommand{\Hcal}{\mathcal H}

\newcommand{\id}{\operatorname{id}}
\newcommand{\Aut}{\operatorname{Aut}}
\newcommand{\GL}{\operatorname{GL}}
\newcommand{\Jac}{\operatorname{Jac}}

\newcommand{\SL}{\operatorname{SL}}

\newcommand{\Isom}{\operatorname{Isom}}

\newcommand{\age}{\mathfrak a}
\newcommand{\Pic}{\operatorname{Pic}}
\newcommand{\Hom}{\operatorname{Hom}}

\newcommand{\Spec}{\operatorname{Spec}}

\newcommand{\al}{\alpha}
\newcommand{\into}{\hookrightarrow}

\newcommand{\ga}{\gamma}

\newcommand{\fie}{\varphi}

\newcommand{\la}{\lambda}

\newcommand{\cxi}{\mathtt{i}}

\newcommand{\FJR}{\operatorname{FJRW}}

\def\pmmu{{\pmb \mu}}

\def\CR{\operatorname{CR}}

\def\ol{\overline}
\def\ul{\underline}
\def\wt{\widetilde}

\begin{document}
\title{Semi-Calabi--Yau orbifolds and mirror pairs}

\author{Alessandro Chiodo}
\address[A.~Chiodo]{{Institut de Mathématiques de Jussieu-Paris Rive Gauche\\
Sorbonne Université - Campus Pierre et Marie Curie \\
4, place Jussieu - Boite Courrier 247\\
75252 Paris Cedex 05, France}}
\email{alessandro.chiodo@imj-prg.fr}

\author{Elana Kalashnikov}
\address[E.~Kalashnikov]{{Department of Mathematics \\ 
Harvard University \\ 
Cambridge, MA 02138}}
\email{kalashnikov@math.harvard.edu}

\author{Davide Cesare Veniani}
\address[D.~C.~Veniani]{{Institut für Algebraische Geometrie \\
Leibniz Universität Hannover\\
 Welfengarten~1, 30167 Hannover, Germany}}
\curraddr{Institut für Topologie und Geometrie \\
Universität Stuttgart\\
Pfaffenwaldring 57\\
70569 Stuttgart\\
Germany}
\email{davide.veniani@igt.uni-stuttgart.de}

%
\begin{abstract}
{We generalize the cohomological mirror duality of Borcea and Voisin in any dimension and for any number of factors. Our proof applies to all examples which can be constructed through Berglund--Hübsch duality. 
Our method is a variant of the so-called Landau--Ginzburg/Calabi--Yau correspondence of Calabi--Yau orbifolds with an involution that does not preserve the volume form. 
We deduce a version of mirror duality for the fixed loci of the involution, which are beyond the Calabi--Yau category and feature hypersurfaces of general type.}
\end{abstract}

\maketitle

\pagestyle{plain}

\section{Introduction}

\subsection{Borcea--Voisin mirror pairs.} 
Nikulin's classification \cite{Ni1} of K3 surfaces $S$
with an anti-symplectic involution $\sigma$
led to a new mirror symmetry statement due to
 Dolgachev \cite{Do-mirror}, Voisin \cite{Vo}
and Borcea \cite{Borcea}.
For any 
$(S,\sigma)$ ,
a mirror partner $(S^\vee, \sigma^\vee)$ is constructed
such that 
crepant resolutions $\wt\Sigma$ and $\wt \Sigma^\vee$
of \[(S\times E)/(\sigma,i)\qquad \text{and}\qquad (S^\vee\times E)/(\sigma^\vee,i)\]
satisfy 
\begin{equation}\label{eq:BVintro} 
H^{p,q}(\wt\Sigma;\CC)\cong H^{3-p,q}(\wt \Sigma^\vee;\CC),\end{equation}
where $E$ is a fixed elliptic curve and $i$ its hyperelliptic involution.
%
This paper generalizes the above duality 
in all dimensions; indeed the above construction holds 
for any even number of factors, and Calabi--Yau orbifolds of any dimension
 at each factor (see Theorem B here below 
 and Theorem \ref{thm:mirror}).
\begin{rem}\label{rem:sloppyCY}
Notice that here, and in the rest of the paper, we refer to 
orbifolds and varieties as ``Calabi--Yau'' with an abuse of terminology: we only insist on the 
property that the canonical bundle is trivial; this is what actually matters for 
the mirror
symmetry constructions considered here.
In particular, we do not have any requirements on the vanishing of $h^{i,0}$, nor do we insist that the fundamental group is trivial, see \S\ref{subsect:BValldim_intro} and Theorem B.
\end{rem} 
 
Our generalization to all dimensions follows almost immediately from a 
refined mirror symmetry statement, 
just as Borcea--Voisin statement \eqref{eq:BVintro} is a consequence 
of the following two facts.
First, the fixed loci 
$S_\sigma$ and $S^\vee_{\sigma^\vee}$ have a cohomological 
mirror behaviour; namely
\begin{equation}\label{eq:fixedlocusMSintro}
H^{p,q}(S_\sigma;\CC)\cong H^{2-p,q}(S^\vee_{\sigma^\vee};\CC).\end{equation}
Second, the anti-invariant and invariant cohomology groups, denoted $H(\ \ ;\CC)^-$ and $H(\ \ ;\CC)^+$ respectively,
satisfy 
\begin{equation}\label{eq:anti-inv_intro} 
H^{p,q}(S;\CC)^{\pm}\cong H^{3-p,q}(S^\vee;\CC)^{\mp}.\end{equation}
For elliptic curves the same properties are trivially satisfied; this explains the appearance of 
the same curve on each side of the mirror. 
We consider a more general setup.

\subsection{Semi-Calabi--Yau models.}
Here, by semi-Calabi--Yau orbifold we mean a
proper and smooth Deligne--Mumford  stack  $\mf Z$ whose 
canonical bundle $\omega_{\mf Z}$ is equipped with a trivialization 
$\Ocal_{\mf Z}\cong \omega_{\mf Z}^{\otimes 2}$.
This yields a $2$-fold étale cover $\pi\colon \mf X\to \mf Z$ 
trivialising $\pi^*\omega_{\mf Z}$. The stack $\mf X$ is equipped with the 
deck involution $\sigma\colon \mf X\to \mf X$, and we recover $\mf Z$ as the stack-theoretic quotient
$$\mf X\to [\mf X/\sigma]=\mf Z.$$


Specifically, our semi-Calabi--Yau\footnote{It is worth mentioning that the derived category $D(\mf Z)$ 
of coherent sheaves on $\mf Z$ with its
Serre functor $\mathsf S_{\mf Z}$ 
is a fractional (semi-)Calabi--Yau category in the sense of Kuznetsov \cite[Def.~1.2]{FCYKuznetsov}: we
have $(\mathsf S_{\mf Z})^2=[2\dim(\mf Z)]$.
} setup arises as follows. 
Let $f$ be a quasi-homogeneous polynomial in 
the variables $x_1,\dots, x_N$ of weight $w_1, \dots, w_N$ and of degree 
$d=2\sum_j w_j$
\begin{equation}\label{eq:degf}
f(\la^{w_1} x_1,\dots, \la^{w_N}x_N)=\la^{2\sum_jw_j}f(x_1,\dots, x_N)\end{equation}
with critical locus reduced to the the origin of $\CC^N$. 
The Cadman--Vistoli square root construction $\mf Z=\PP(\pmb w)_{\Ocal(d),f,2}$ 
is a stack $\mf Z$ for which there exists 
a morphism $p\colon \mf Z\to \PP(\pmb w)$
with a 
line bundle $M$ and an isomorphism $M^{\otimes 2}\to p^*\Ocal(d)$. Its canonical 
bundle $\omega_{\mf Z}$ equals $p^*\omega_{\PP(\pmb w)}\otimes M$ whose square 
is equipped with the trivialization 
\[\omega_{\mf Z}^{\otimes 2}=p^*\omega_{\PP(\pmb w)}^{\otimes 2}\otimes M^{\otimes 2}=
p^*(\omega_{\PP(\pmb w)}^{\otimes 2}(d))\cong \Ocal_{\mf Z}.\]
The corresponding étale double cover of $\mf Z$ can be realized as the stack 
\[\mf X=\{x_0^2+f(x_1,\dots,x_N)=0\}\subset \textstyle{\PP\left(\frac{d}2,w_1,\dots,w_n\right)}\]
with the involution $\sigma\colon (x_0,x_1\dots,x_N)\mapsto (-x_0, x_1,\dots,x_N)$.
In this context, Theorem A below 
applies to 
mirror pairs defined by an explicit construction due to 
 Berglund and Hübsch \cite{BH}, and 
is the generalized version of \eqref{eq:fixedlocusMSintro} and 
\eqref{eq:anti-inv_intro}. 
It applies more generally to the 
equivariant semi-Calabi--Yau 
\begin{equation}\label{eq:semiprototype} \mf X=[\{x_0^2+f(x_1,\dots,x_N)=0\}/{H_0}]\to [\mf X/\sigma]=\mf Z. \end{equation}
Here 
$H$ is a group of diagonal morphisms $\diag(\al_0,\dots, \al_N)$ 
of determinant $1$
preserving 
the polynomial $x_0^2+f$. The group 
$H_0$ is 
 the quotient of $H$ by the subgroup of 
actions of the form $(\la^{d/2}, \la^{w_1},\dots,\la^{w_N})$ with $\la \in \GG_m$.

\subsection{Berglund--Hübsch mirror duality.}
The mirror Theorem A below relates two semi-Calabi--Yau's of the form 
\eqref{eq:semiprototype} attached to two pairs $(f,H)$, where $f$ and $H$ are a polynomial  and a group as above.
We assume 
$f(x_1,\dots, x_N)=\sum_{j=1}^N x_j^{m_{i,j}},$
with $m_{i,j}\in \NN$ and $M=(m_{i,j})$ invertible.
Then, by transposing $M$, we set 
\[f^\vee(x_1,\dots, x_N)=\sum_{j=1}^N \prod_{j=1}^N x_j^{m_{j,i}},\]
and we have a canonical isomorphism $\Hom(\Aut_{\diag}(x_0^2+f);\GG_m)\cong 
\Aut_{\diag}(x_0^2+f^\vee)$, where $\Aut_{\diag}(P)$ denotes the group of all
diagonal symmetries preserving a polynomial $P$.
In this way, for each subgroup $H\into \Aut_{\diag}(x_0^2+f)$ as above, we 
set 
\[H^\vee=\ker \left(\Aut_{\diag}(x_0^2+f^\vee)\twoheadrightarrow \Hom(H;\GG_m)\right).\]
We assume that $H$ contains $(\la^{d/2}, \la^{w_1},\dots,\la^{w_N})$
for $\la=\xi_d$, \emph{i.e.} the monodromy operator of the fibration $x_0^2+f$ over $\CC^\times$; then, all the relevant properties are preserved by the duality $(f^\vee,H^\vee)$: the only critical 
point of $f^\vee$ is the origin, \eqref{eq:degf} holds for $f^\vee$, and $H^\vee$ 
is formed by diagonal matrices of determinant $1$ (see \S\ref{sect:dualpolys}).
Before stating the generalized version of \eqref{eq:fixedlocusMSintro} and 
\eqref{eq:anti-inv_intro}, let us specify the relevant orbifold cohomology 
of a fixed locus within a stack.

\subsection{Orbifold cohomology classes depending on an automorphism.}
Orbifold Chen--Ruan cohomology groups $H_{\CR}^*(\mf X ;\CC)$ 
of a smooth Deligne--Mumford stack $\mf X$ are 
the cohomology groups of the fiber product 
\[\mathfrak X{\underset{\id, \ \mathfrak X\times \mathfrak X, \ \id}{\times}}\mathfrak X\]
via the graph of the identity morphism, \emph{i.e.} the diagonal (see definitions \ref{defn:sigmainertiastack} and \ref{defn:sigmainertiastack}). 
The grading is obtained after a shift with 
respect to the locally constant ``age'' function, see \S\ref{subsect:age}. 
Whenever the stack is Gorenstein (\emph{i.e.} the stabilizers locally operate 
with determinant $1$) and a crepant resolution $\wt X$ of 
the coarse space $X$ of $\mf X$ exists, 
these orbifold cohomology groups are isomorphic to the 
ordinary cohomology $H^*(\wt X;\CC)$. 

We generalize the definition and introduce 
$\sigma$-orbifold cohomology classes $H_\sigma^* (\mf X;\CC)$ as the cohomology of 
\[\mathfrak X{\underset{\sigma, \ \mathfrak X\times \mathfrak X, \ \id}{\times}}\mathfrak X\]
with respect to the graph of an automorphism $\sigma\colon \mf X\to \mf X$ (see Definition \ref{defn:sigmainertiacohomology}). 
This is a bi-graded group as above, with age-shifted grading, see \S\ref{subsect:age}.
For the Calabi--Yau orbifolds studied here, we prove that, for $\dim(\mf X)=2$ and $\sigma$ anti-symplectic,
the cohomology of the fixed point set $\wt X_\sigma$ of the minimal resolution $\wt X$ of the Gorenstein 
coarse space $X$
satisfies \begin{equation}\label{eq:crepantintro}
H^{p,q}(\wt X_\sigma;\CC)\cong H_\sigma^{p+\frac12,q+\frac12}(\mf X;\CC),\end{equation}
see Proposition \ref{pro:crthm}. We refer to Proposition \ref{pro:alldim} for 
a generalization in all dimensions conditional to the existence of a crepant resolution and a lift of 
the involution $\sigma$.

\subsection{A mirror symmetry theorem for semi-Calabi--Yau orbifolds.}
We finally state the refinement of the ordinary cohomological mirror
theorem. 

\medskip 

\noindent \textbf{Theorem A (Semi-Mirror Theorem, Thm.~\ref{thm:semimirror}).} \emph{
Let $(f,H)$ and $(f^\vee, H^\vee)$ be two polynomials as above. 
Consider the quotient stacks $\mf X$ and $\mf X^\vee$ defined 
as the vanishing locus of $x_0^2+f$ and $x_0^2+f^\vee$ 
modulo $H_0$ and $(H^\vee)_0$ (with $\sigma$-involution).
We have }
\begin{enumerate}[(i)]
\item
$H^{p,q}_{\CR}
(\mf X;\CC)^\pm\cong
H^{n-1-p,q}_{\CR}(\mf X^\vee;\CC)^\mp;$
\item
$H^{p,q}_\sigma(\mf X;\CC)\cong
H^{n-1-p,q}_\sigma(\mf X^\vee;\CC).$
\end{enumerate}

\medskip 

\noindent Via \eqref{eq:crepantintro}, the above result specializes to 
\eqref{eq:fixedlocusMSintro} and \eqref{eq:anti-inv_intro}.
See Corollary \ref{cor:crepantsemi} for a statement in all dimensions.

\subsection{Borcea--Voisin duality in any dimension.}\label{subsect:BValldim_intro}
A direct consequence of the above semi-mirror Theorem A
is the ordinary mirror symmetry duality of Borcea--Voisin type 
in any dimension. We refer the reader to 
the statement of Theorem \ref{thm:mirror}
for a more general statement involving several group quotients of 
the stack $\prod_{i=1}^n \mf X_i$. 

\medskip

\noindent \textbf{Theorem B (Borcea--Voisin Mirror Theorem, Thm.~\ref{thm:mirror}).}
\emph{For any $i=1,\dots, 2n$, let $(f_i,H_i)$ be pairs as above 
defining an $m_i$-dimensional stack $\mf X_i$ with involution $\sigma_i$. 
Then we have }
\[H_{\CR}^{p,q}\left(\left[\prod_i \mf X_i/(\sigma_1,\dots, \sigma_{2n})\right];\CC\right)\cong 
H_{\CR}^{\sum_i m_i -p,q}\left(\left[\prod_i \mf X_i^\vee/(\sigma_1,\dots, \sigma_{2n})\right];\CC\right).\]

\medskip 

\noindent 
This statement turns into an ordinary cohomology statement whenever crepant resolutions on the two sides exist.
To the best of our knowledge, this mirror symmetry construction 
is new. It also provides many examples of orbifold Calabi--Yau in our weak sense of 
geometric object (orbifold or varieties) for which $\omega$ is trivial (see Remark \ref{rem:sloppyCY}). It is easy to see that the orbifolds do not necessarily satisfy $h_{\CR}^{i,0}=0$ if $i\neq 0$ or $i \neq \dim$; for instance, consider
the case of the product of four Fermat elliptic curves modulo the above involution yielding $h^{2,0}_{\CR}=6$ (and $h_{\CR}^{0,0}=h_{\CR}^{4,0}=1, h_{\CR}^{1,1}=h_{\CR}^{3,1}=16, h_{\CR}^{2,2}=36$). 

\subsection{The proof via unprojected Landau--Ginzburg models.}
The Jacobi ring of a singularity 
has a natural orbifold version known as the FJRW 
or Landau--Ginzburg state space $H^*_{\mathrm{FJRW}}(x_0^2+f,H)$. It was proven by the first author and Ruan \cite{ChiRuLGCY} that $H^*_{\CR}(\mf X;\CC)$ and $H^*_{\mathrm{FJRW}}(x_0^2+f,H)$ are isomorphic if $\omega_{\mf X}\cong \Ocal_{\mf X}$ (LG/CY correspondence). 
We provide a $\sigma$-orbifold version
by recasting the FJRW states space into 
an ``unprojected'' Landau--Ginzburg (LG) state space (see \eqref{eq:unproj}) $
\mathcal H_K(x_0^2+f)^H\supseteq H^*_{\mathrm{FJRW}}(x_0^2+f,H)$ already 
considered by Krawitz~\cite{Krawitz} (the $H$-invariant 
Jacobi ring orbifolded on $K=H[\sigma]$).

The proof can now be carried out in terms of this unprojected LG model, 
that, under the LG/CY correspondence, embodies three invariants 
\begin{enumerate}
\item $\sigma$-invariant classes of $\mf X$;
\item $\sigma$-anti-invariant classes of $\mf X$;
\item $\sigma$-invariant $\sigma$-orbifold classes of $\mf X$;
\end{enumerate}
(there are no $\sigma$-anti invariant $\sigma$-orbifold classes of $\mf X$ 
as the reader may expect from \eqref{eq:crepantintro}, 
which relates $\sigma$-orbifold classes 
to the fixed point set of the resolution).
 Under mirror symmetry the groups $H$ and $K$ switch: 
\[ \mathcal H_K(x_0^2+f)^H=\mathcal H_{H^\vee}(x_0^2+f^{\vee})^{K^\vee}. \]
Unfortunately, the LG/CY correspondence does not apply to 
the unprojected state space on the right hand side (this happens because the group 
duality reverses the inclusions and yields a too small group $K^\vee$). However, 
we can remedy to this, after 
a simple isomorphism (see Lemma \ref{lem:strippingtwisting}), a contraction of the form 
 \begin{equation}\textstyle\label{eq:adhoc}\fie(x_1,\dots, x_N) dx_0\wedge \bigwedge_{i=1}^N dx_i \mapsto \fie(x_1,\dots, x_N) \bigwedge_ {i=1}^N dx_i \mid _{\CC^N}.\end{equation}
Ultimately LG/CY correspondence can be applied and we notice that 
mirror symmetry operates an exchange of 
lines $(1)$ and $(2)$ and maps $(3)$ to its mirror analogue. This is the semi-mirror Theorem A above.

\subsection{Berglund--Hübsch mirror symmetry for K3 surfaces.}
In \cite{ABSBHCR}, Artebani, Boissière and Sarti considered the case of K3 surfaces arising 
from Berglund--H\"ubsch mirror symmetry, and checked that Berglund--H\"ubsch duality 
is consistent with the mirror symmetry construction based on Nikulin's classification. 
Nikulin's classification can be phrased in terms of the invariants 
$h^{0,0}(S_\sigma)$ and $h^{1,0}(S_\sigma)$, and a third invariant $\delta\in \ZZ/2$ 
vanishing if and only if $[S_\sigma]\in 2H^*(S;\ZZ)$. Artebani--Boissière--Sarti's check 
consist in proving that $h^{0,0}(S_\sigma)$ and $h^{1,0}(S_\sigma)$ are exchanged 
and that the property $[S_\sigma]\in 2H^*(S;\ZZ)$ is preserved. The 
first claim is a corollary of the Semi-Mirror Theorem and \eqref{eq:crepantintro}. Since \cite{ABSBHCR}
 relies on explicit case-by-case resolution, our main result simplifies 
their proof a great deal. 
As far as the property $[S_\sigma]\in 2H^*(S;\ZZ)$ goes, its conservation 
under mirror symmetry does not appear to follow from our LG/CY methods. 

\subsection{Other related works.}
In his early paper \cite{Borcea}, Borcea already highlighted 
the importance of properties \eqref{eq:fixedlocusMSintro} and
\eqref{eq:anti-inv_intro}. In \cite[\S2, \S10]{Borcea} (``Higher dimensions''), 
he went further to 
consider mirror pairs of Calabi--Yau varieties with involutions in higher dimension, 
and to check that 
the Euler characteristics $\chi(S),\chi(S/\sigma), \chi(S_\sigma)$ 
all change by $(-1)^{\dim(S)},(-1)^{\dim(S)}, (-1)^{\dim(S)-1}$ under mirror symmetry in dimension 3 and 4,
as one can now deduce from the Semi-Mirror Theorem in Berglund--H\"ubsch setup. 
In \cite{Dillies}, 
this approach is pushed further by Dillies 
to a proof of cohomological mirror symmetry for crepant resolutions of dimension $3$ and $4$. 
Crepant resolutions of quotients of 
products of Calabi--Yau with an involution 
fixing a smooth divisor are provided by Cynk and Hulek's work \cite{CyHu}.
The case of higher-order automorphisms is considered in 
\cite{Dillies} as well as \cite{ComparinPriddis} and \cite{CLPS}.
Propositions \ref{pro:crthm} and \ref{pro:alldim} are related to
Ruan's Crepant resolution conjecture \cite{RuanCRC}. 
Finally, very recently, Hull, Israel and Sarti used mirror symmetry for K3 surfaces to form ``non-geometric backgrounds'' in the physics paper \cite{HIS}.

\subsection{Contents}
In \S\ref{sect:terminology}, we recall terminology briefly. 
In \S\ref{sect:preliminaries}, we recall some basic definitions about Berglund--Hübsch invertible polynomials.
In \S\ref{sect:orbifolds}, we treat orbifold cohomology, its $\sigma$-orbifold variant, and we prove the compatibility result \eqref{eq:crepantintro} stated above.
In \S\ref{sect:LG-models}, we prove all the relevant statements at the level 
of Landau--Ginzburg state spaces. 
In \S\ref{sect:involutions}, we derive the corresponding geometric versions stated above, 
see in particular \S\ref{subsect:geomMS} with some examples. 
Relation to K3 surfaces is treated in \S\ref{subsection:ABS}; we compare to the 
approach of \cite{ABSBHCR} in Example \ref{exa:inhomog}.
The higher dimensional Borcea--Voisin mirror theorem is deduced in \S\ref{sect:generalized-BV}. 

\subsection*{Acknowledgements}
We are grateful to Alfio Ragusa, Francesco Russo and Giuseppe Zappalà for organizing 
Pragmatic 2015, where this work started. 
We are grateful to Tom Coates, Baohua Fu, Lie Fu, Matthieu Florence, Behrang Noohi, Alessandra Sarti, Angelo Vistoli, Claire Voisin, and Takehiko Yasuda for their advice. This work was finalized at ETH, Zurich, and Imperial College, London; the first named author 
thanks these institutions for their hospitality. 

The work of the second author was supported by the Engineering and Physical Sciences Research Council [EP/L015234/1],  and the EPSRC Centre for Doctoral Training in Geometry and Number Theory (The London School of Geometry and Number Theory), University College London. 

The third author would like to acknowledge the support of the GRK 1436 "Analysis, geometry and string theory". 

Finally, we are extremely grateful to the anonymous referee for his careful reading and for his precious comments which helped improving and making much more precise the main statements of the paper.

\section{Terminology} \label{sect:terminology}

\subsection{Conventions}
We work with schemes and stacks over the complex numbers. All schemes are Noetherian and separated.
By linear algebraic group we mean a closed subgroup of 
$\GL_m(\CC)$ for some $m$.
We often need to identify 
a stack locally. 
In order to avoid repeated mention of 
\'etale localization or strict Henselizations, we often use the expression 
``the 
local picture of the stack $\mf X$ at the geometric point $x\in \mf X$ is the same as $\mf U$ at $u\in \mf U$''. 
By this we mean that the strict Henselization of $\mf X$ at $x$ is the same of that 
of $\mf U$ at $u$. 
Often it is enough to say that there is an \'etale neighbourhood $X'$ of $x$ 
 and an isomorphism $X'\to U'$ with an \'etale neighbourhood $U'$ of $u$.
We refer to \cite[54.33.2]{stacksproject} for a definition of the 
strict Henselization 
and to \cite[\S1.2,5]{ACV} for further discussion (see in particular the ``algebra-to-analysis translation'', where strict Henselizations 
are described analytically as the germ of $X$ at $x$).

\subsection{Notation} We list here notation that occurs throughout the entire paper. 

\vspace{5pt} 

\begin{tabular}{ll}
$V^K$& is the invariant subspace of a vector space $V$ linearized by a finite group $K$.\\
$\PP(\pmb w)$ & is the quotient stack $[(\CC^n\setminus \pmb 0)/\GG_m]$, where $\GG_m$ acts with weights $\pmb w$.\\
$Z(f)$ & is the variety defined as 
 zero locus of $f\in \CC[x_1,\dots, x_n]$.\\
\end{tabular}

\begin{rem}[zero loci]\label{rem:zeroloci} We add the subscript 
$\PP(\pmb w)$ when we refer to the zero locus in $\PP(\pmb w)$ of a polynomial $f$ which is $\pmb w$-weighted homogeneous. In this way we have
\[Z_{\PP(\pmb w)}(f)=[U/\GG_m],\qquad \text{\ \ with $U=Z(f)\subset \CC^n\setminus \pmb 0$.}\]
\end{rem}

\begin{rem}[graphs and maps]\label{rem:abuse_notn_graph}
Given an automorphism $\alpha$ of $\mathfrak X$, 
we write $\Gamma_\alpha$ for 
the graph $\mathfrak X\to 
\mathfrak X\times \mathfrak X$. However, to simplify 
formul\ae, 
we often abuse notation 
and use $\alpha$ for the graph $\Gamma_\al$ as well as the automorphism. In this way, the diagonal 
$\Delta \colon \mathfrak X\to \mathfrak X\times \mathfrak X$ will be often written as $\id_{\mathfrak X}$ or simply $\id$.
\end{rem}

\section{Berglund--Hübsch polynomials} \label{sect:preliminaries}

The setup presented here is due to Berglund--H\"ubsch \cite{BH}. We also refer to 
\cite{Berglund-Henningson,GZE_Saito, Gusein-Zade-Ebeling, KreSka, Krawitz}. It can be motivated 
as the simplest generalization of Greene--Plesser mirrors.
\subsection{Invertible polynomials} \label{sect:invertibleW}
Let
\begin{equation}\label{eq:invertibleW}
W(x_1,\dots,x_n)=\sum_{i=1}^n \prod_{j=1}^n x_j^{m_{i,j}},\end{equation}
be a quasi-homogeneous polynomial of weights $w_1,\dots,w_n$ 
and degree $d$. 
The polynomial is said to be \emph{invertible} if the matrix $M =(m_{i,j})$ admits an inverse 
$M^{-1}=(m^{i,j})$. 
We could more naturally start from a polynomial of 
the form $\sum_{i=1}^n c_i\prod_{j=1}^n x_j^{m_{i,j}}$ (with $c_i\neq 0$), but after rescaling the variables $x_j$ 
we can reduce to the above case, without loss of generality. 
We always assume $W$ to be \emph{non-degenerate}, \emph{i.e.}, regarded as a complex valued function, we have $\partial W(x_1,\dots,x_N)/\partial x_j=0$ for every $j$ only at $(x_1,\dots,x_N) = (0,\dots,0)$. 

Non-degeneracy is a very restrictive condition, and complete classification of non-degenerate polynomials is given in \cite{KreSka} (see also \S5 and Theorem~5.2 in \cite{FavKel}). 
We do not use this classification, but we recall it briefly. 
After permutation of the variables, the matrix necessarily decomposes into irreducible $1\times 1$ blocks within $\ZZ_{\ge 1}$ (Fermat blocks) and blocks 
of the form $k\times k$ (with $k>1$) with $a_{i,j}=0$ for $j-i\not\in \{0,1\}+k\ZZ$,
$a_{i,i}\in \ZZ_{\ge 1}$ 
($1\le i \le k$), $a_{i,i+1}=1$ ($1\le i < k$), and $a_{k,1}=1$ (loop blocks) or $a_{k,1}=0$ (chain blocks). The polynomials corresponding to the blocks described above are usually referred to as Fermat, loop, and chain polynomials. 

\subsection{Calabi--Yau varieties}
The \emph{charge} of the variable $x_i$ is defined as the ratio $q_i := w_i/d$; it is uniquely determined by $W$, as the sum $q_i = \sum_j m^{i,j}$ of the entries of the $i$th line of $M^{-1}$.
We say that $W$ satisfies the \emph{Calabi--Yau condition}, or that $W$ is a Calabi--Yau invertible polynomial, if we have
\begin{equation}\label{eq:CY}
{\textstyle{\sum_{j}}} w_j=d,
\end{equation}
or, equivalently, if the sum of all the entries of $M^{-1}$ is $1$.

\begin{rem}\label{rem:indices} The set of data $(w_1,\dots,w_n;d)$ is uniquely determined by $W$ as soon as we reduce these indices so that $\gcd(\pmb w)=1$. Note that the Calabi--Yau condition implies $\gcd(\pmb w)=\gcd(\pmb w,d)$.
\end{rem}

\begin{rem}\label{rem:quasismooth} The non-degeneracy condition is equivalent to the smoothness of the 
vanishing locus
$Z_{\PP(\pmb w)}(W)$ of $W$ within 
the stack $\PP(\pmb w)$. The coarse space of 
the hypersurface 
$Z_{\PP(\pmb w)}(W)$, within the coarse space 
of $\PP(\pmb w)$, may be singular but quasi-smooth 
in the sense of \cite[App. B]{Dimca} (see 
Remark \ref{rem:Dimca}). 
\end{rem}

By the adjunction formula, condition \eqref{eq:CY} is equivalent to 
the fact that the canonical bundle of 
$Z_{\PP(\pmb w)}(W)$ is trivial. 
This justifies the term ``Calabi--Yau'',
and provides an important 
source of examples of Calabi--Yau orbifolds, yielding
Calabi--Yau varieties whenever there exists a 
crepant resolution $f\colon \wt Z\to Z_{\PP(\pmb w)}(W)$ (\emph{i.e.} with $f^*\omega=\omega_{\wt Z}$).
This occurs for instance in dimension $\le 3$.

\subsection{Finite order diagonal actions}\label{sect:notnfiniteorderaction}
Let $\al\in \GL_m(\CC)$ be an $m\times m$ diagonal matrix of finite order with complex coefficients. The entries along the diagonal 
are necessarily roots of unity; for sake of simplicity, we write
\begin{equation}\label{eq:Reidnotn}
\al = (a_1,\dots,a_m), \quad a_j \in \QQ, \quad 0 \leq a_j < 1
\end{equation}
if the $j$-th diagonal entry of the diagonal matrix $\al$ is $\exp(2\pi\cxi a_j)$. Each $\alpha_j$ is uniquely determined and the \emph{age} of~$\al$ is defined as 
\[	\mathrm{age}(\al):={\textstyle \sum_j }a_j. \]

For any polynomial $f = f(x_1,\ldots,x_m)$ in $m$ variables and for any 
$\al$ acting diagonally on the domain of $f$, we denote by 
$f_\al$ the restriction to the fixed space $\CC^m_\al$ spanned 
by the fixed variables $x_j\mid \al^*x_j=x_j$. We 
often use the set of labels of the fixed variables, and we 
denote it by 
\begin{equation*}\label{eq:indicesfixedvariables}
F_\al=\{j\mid \al^*x_j=x_j\};\qquad 
\CC^m_\al=\Spec \CC[x_j\mid j\in F_\al].\end{equation*}

Given an invertible polynomial $W$ as in \eqref{eq:invertibleW}, let $\Aut_W=\Aut_{\mathrm{diag},W}$ be the group of 
diagonal matrices $\al$ such that $W(\al^*\pmb x)=W(\pmb x)$. 
The fact that $\Aut_W$ is finite is a consequence of the invertibility of the matrix $M=(m_{i,j})$: regard $M$ as a linear map $\QQ^n\to \QQ^n$ so that 
$\Aut_W$ is the quotient of the 
rank-$n$ lattice $M^{-1}\ZZ^n$ by the sublattice $\ZZ^n$.

We also consider 
\[\SL_W:=\SL_n(\CC) \cap \Aut_W,\]
and the order-$d$ group generated by 
the so-called \emph{grading element} of $\Aut_W$
\[j_W:=(q_1,\dots,q_n).\]
Without mentioning the charges $q_j$, this can be defined as the monodromy operator of the fibration 
$W\colon \CC^n\setminus W^{-1}(0)\to \CC^\times$.

\subsection{A combinatorial reinterpretation} Although this plays no relevant role in the statements of this paper, it is worth pointing out that 
the above data $(m_{i,j})$ may be phrased as follows. 
The matrix $(m_{i,j})$ is an integer, non-degenerate pairing
between two lattices $E=\bigoplus_i e_i\ZZ $ and $F=\bigoplus_j f_j \ZZ$ with $\langle e_i, f_j\rangle=m_{i,j}\in \ZZ\ge 0$. In this way $F$ (resp. $E$) is a rank-$N$ sublattice of $E^\vee$ (resp. $F^\vee$). 
As mentioned above, the 
group of diagonal automorphisms $\Aut_W$ is merely 
the quotient $E^\vee/F$. 

\begin{rem}\label{rem:lattices}
The injective map $F\to E^\vee,\ f_j\mapsto \langle \_, f_j\rangle$ is represented by $M=(m_{i,j})$
and the map $E\to F^\vee,\ e_i\mapsto \langle e_i,\!\_\rangle$ is represented by the transpose $M^T=(m_{j,i})$. This yields a canonical automorphism between 
the group of characters $(\Aut_W)^*=\Hom(\Aut_W, \GG_m)$ and the above group of diagonal 
automorphisms relative to the polynomial whose exponents 
are given by the transpose matrix $M^T=(m_{j,i})$.
We refer to \cite{Berglund-Henningson} and \cite[Prop. 2]{GZE_Saito}.

We will restate and 
rephrase again this transposition property
when we will introduce mirror symmetry in 
Section \ref{sect:LG-models}. 
\end{rem}

\begin{rem}
The setup presented here naturally 
yields a reformulation in toric geometry. 
We refer to \cite{BorisovBH} and \cite{FavKelDerCat}.
\end{rem}

\section{Orbifold cohomology classes}\label{sect:orbifolds}
\subsection{Background hypotheses} 
We provide a presentation of orbifold cohomology classes with some generalizations to the standard setup. 
As a special case, we recall the definition of Chen--Ruan cohomology groups.
We recall that a Deligne--Mumford orbifold 
is a smooth separated Deligne--Mumford stack with a
dense open subset isomorphic to an algebraic variety. We work under the following assumption. 
\begin{assume}\label{background_assumption} We consider
a Deligne--Mumford orbifold $\mf X$ equipped with an automorphism 
$\sigma\colon \mf X\to \mf X$ of finite 
order $h$. 
\end{assume}

We often specialize our treatment to quotient stacks $[U/G]$, where $G$ is a linear algebraic group. An algebraic group $G$ acts properly 
on a scheme $U$ if the map $G\times U\to U\times U,\ (g,x)\mapsto(x,g\cdot u)$ is proper. In general, the stack $[U/G]$ 
is a separated Deligne--Mumford stack if and only if $G$ acts properly
on $U$ (see for instance \cite[\S2.1]{logtraceEJK}). 
Since $G$ is affine, the map $G\times U\to U\times U$ is finite if it is proper. 
In particular, the stabilizers $G_x=\{g\mid g\cdot x=x\}$ are finite for all $x\in U$.

When in the special case of a quotient stack as above, we will usually make the further assumption as follows.
\begin{assume}[Quotient stacks]\label{specific_assumption}
Given Assumption \ref{background_assumption},
we further assume that $\mf X$ admits a presentation as a 
quotient stack $[U/G]$, where $U$ is a smooth scheme and $G$ is a linear algebraic group acting properly on $U$. Moreover, we assume that the automorphism of $[U/G]$ is 
represented by an automorphism of schemes 
$\sigma\colon U\to U$ and that 
$G$ and $\sigma$
satisfy $\sigma G\sigma^{-1}=G$ and $\sigma^h\in G$ within 
a linear algebraic group acting properly on $U$. 
\end{assume}

\begin{rem} An automorphism $t$ of 
$\mf X=[V/H]$ of order $h$ does not necessarily lift to $V$. However, if 
$\mf X$ can be represented as a quotient stack, then it is always possible to find a representation such that Assumption \ref{specific_assumption} holds. 
Namely, there always exists 
a scheme $U$, a group $G$ operating on $U$ so that $[U/G]=[V/H]$, and 
an automorphism $\sigma$ of $U$ and $G$ inducing the automorphism $t$. 
For instance, as was suggested to us by Angelo Vistoli, let $U$ be 
the fiber product of the schemes $V_i=(t^i)^*V$ over $\mf X$ for $i=0,\dots, h-1$ and
let $G$ be $H^h$ with the obvious action on $U$. 
Then, $\sigma$ operates on $U$ and, compatibly, on $G$: it is 
the cyclic permutation of the $h$ factors of $U$ and of the factors
of $G$. In this way $\sum_{i=0}^{h-1}\sigma^iG$ can be naturally turned into a group acting properly acts on $U$ and, within it, $G$ and $\sigma$ satisfy 
$\sigma G\sigma^{-1}=G$ and $\sigma^h\in G$.
\end{rem}

\subsection{Applications}\label{subsect:applications}
The subject of this paper are Deligne--Mumford stacks $\Sigma_{W,G}$ 
presented as quotients 
of a scheme 
$U\subseteq \CC^n\setminus \pmb 0$ ($U$ is the locus where $W$ vanishes) 
by a properly acting group $K$ (see Section \ref{sect:orbdata} for full definition). In these cases, we can lift $\sigma$ to $U$ directly
and naturally fit the hypotheses of Assumption \ref{specific_assumption}. 
The group $K$ is an extension of a finite group of diagonal 
automorphisms in $\GL(n;\CC)$ by the above mentioned 
weighted representation of $\GG_m$ 
\[ \fie_{\pmb w}\colon \GG_m\to \GL(n;\CC),\ \ \la \mapsto (\la^{w_1},\dots,\la^{w_n}), \]
where $\pmb w$ is a vector of $n$ coprime positive integers. 
We notice that, in this case, 
the fixed spaces $U_g=\{x=(x_1,\dots, x_n)\in U\mid g\cdot x=x\}$ are non-empty only for a finite number 
of elements $g\in K$. The scheme $U$ possesses an involution $\sigma \colon U\to U$ given by a diagonal automorphism in $\GL(n;\CC)$ acting by multiplication by $-1$ on the first coordinate. 
Since $\sigma$ commutes with all diagonal automorphisms, we immediately have that
$G$ and $\sigma$ satisfy $\sigma G\sigma^{-1}=G$ within 
$G(\sigma)$, which acts properly on $U$. 
 
Second, in Theorem \ref{thm:mirror}, we need to treat the product of 
$r$ factors of type $\Sigma_{W,G}$ modulo products of involution of the type 
described above. 
Again we can present this in the setup of quotient stacks $[U/G]$ 
easily with the involutions lifting 
directly on $U$. 
For $r>1$ and $n=n_1+\dots+n_r$, 
we consider the vanishing locus $U$ of several 
non-degenerate polynomials $W_i(x_{0,i},\dots,x_{n_i,i})$ 
within $$\CC^{n}-
Z=\{(x_{1,1},\dots,x_{n_1,1},\dots,x_{1,r},\dots, x_{n_r,r}) \mid x_{1,i} = \cdots =x_{n_i,i}=0 \text{ for some } i\}.$$ 
We consider a group $K$, which is the 
extension of a finite group of diagonal 
automorphisms in $\GL(n;\CC)$ by the group 
$$(\fie_{\pmb w_1},\dots,\fie_{\pmb w_r})\colon (\GG_m)^r\into \GL(n;\CC)$$
(each index $\pmb w_i$ is a vector of $n_i$ coprime positive coordinates). 
The group $K$ operates on 
$U\subseteq \prod_{i=1}^r (\CC^{n_i}\setminus \pmb 0)$. 
Again, the involutions we want to consider naturally lift to the scheme $U$. 
Again, notice that $U_g$ is non-empty only for a finite number 
of elements $g\in G$. 

\subsection{Inertia stacks}\label{sect:thestackweconsider}
Our discussion of orbifold cohomology classes 
will require two ingredients usually referred to as the ``inertia'' and the ``age''.
Inertia constructions are natural geometric objects keeping track of geometric points and elements of their stabilizers. 
The age is a locally constant, positive, $\QQ$-valued function defined on them.

We work under Assumption \ref{background_assumption}; \emph{i.e.} $\mf X$ is a Deligne--Mumford stack.
The inertia stack is a fiber product
\begin{equation}\label{eq:inertiastack}\mathfrak I_{\mathfrak X}:=\mathfrak X{\underset{\id, \ \mathfrak X\times \mathfrak X, \ \id}{\times}}\mathfrak X,\end{equation} 
with respect to the diagonal morphism $\mathfrak X \to \mathfrak X\times \mathfrak X$, which is denoted by $\id$ instead of $\Gamma_{\id}$ as mentioned above.
The stack $\mathfrak I_{\mathfrak X}$ is a category whose objects over 
a scheme $T$ are 
pairs $(\ga,\xi)$, where $\xi$ is an object of $\mf X$ over $T$ 
and $\ga$ is an element of $\Aut_T(\xi)$; these objects form a groupoid 
whose isomorphisms are given by $(\ga, \xi)\to 
(\al \ga \al^{-1}, \al \xi)$ for any automorphism 
$\al\in \Aut_T(\xi)$.

For $T=\Spec \CC$, this allows us to describe the 
geometric points of the inertia stack as pairs $(g,x)$ 
given by geometric points $x$ and automorphisms $g$ of $x$ up to $(g,x)\cong(\al g \al^{-1},\al x=x)$. In this way, the fiber over a geometric 
point $x\in \mf X$ is a disjoint union of stacks of the 
form $BH$ in one-to-one correspondence with 
the conjugacy classes of $G=\Aut(x)$ with $H$ equal to the 
centralizer of each class.

For a quotient stack $\mathfrak X= [U/G]$, 
where $U$ is a smooth scheme and $G$ is a 
linear algebraic group acting properly on $U$,
we can provide a presentation of the inertia stack 
in terms of quotient stacks
\[\mathfrak I_\mathfrak X=[I_G(U)/G],\]
where $I_G(U)$ and the action of $G$ are 
recalled below following the statements of
\cite[Lem.~70.25.1]{stacksproject}, 
which apply to the more general framework of groupoid schemes (in \cite{stacksproject} the group scheme $I_G(U)\to U$ is defined as the ``stabilizer of the groupoid''; the reader can also find 
the direct definition 
of $I_G(U)$ with its $G$-action in
\cite[\S2.2]{logtraceEJK}).

For any closed subscheme $S$ in $G$, 
the $S$-inertia $U$-scheme
\[I_S(U)=\{(g,x)\in S\times U\mid g\cdot x=x\},\]
can be realized as the base change of 
$(G\text{-action},\id_U)\colon S\times U\to U\times U;$ via the diagonal 
$U\to U\times U$
\[I_S(U)=(S\times U)\times_{U\times U} U.\]
Since $S$ is a closed subscheme of $G$, and 
$G$ acts properly, the scheme $I_S(U)$ is finite over $U$ (we refer to \cite[Defn.~2.3, (i) and (ii)]{logtraceEJK}).

The group $G$ operates on $I_G(U)$ by conjugation
on the first factor and by multiplication on the left on the second factor. Since
$g\cdot x=x$ we have $\al g\al^{-1}\cdot \al x=\al x$. 

\begin{rem} \label{rem:slightgeneralizationS} 
The 
action makes sense on $I_S(U)$ as soon as 
$\al S \al ^{-1}=S$ for any $\al \in G$. Below, we 
use the construction $I_S(U)$ for $S\neq G$ for 
a slight generalization: the $\sigma$-inertia stack.\end{rem}


\begin{defn}[$\sigma$-inertia stack]\label{defn:sigmainertiastack}
For any automorphism $\sigma \colon \mf X\to \mf X$ of finite order, the $\sigma$-inertia stack is given by
\begin{equation}\label{eq:sigmainertiastack}
\mathfrak I^\sigma_{\mathfrak X}:=\mathfrak X{\underset{\sigma, \ \mathfrak X\times \mf X, \ \id}{\times}}\mathfrak X,\end{equation}
where, to simplify notation, $\sigma$ is the graph 
$\sigma\colon \mathfrak X\to \mathfrak X\times \mathfrak X.$
\end{defn}

The automorphism $\sigma$ is a functor on $\mf X$. 
To each object 
$\xi\colon T\to \mf X$ of $\mf X$, we associate $\sigma(\xi)$, 
the composite morphism $\sigma \xi$.
Each morphism $\al$ from $\xi\colon T\to \mf X$ to $\nu\colon S\to \mf X$ 
is a morphism $S\to T$ commuting with $\xi$ and $\nu$. Given such a morphism $\al$,
we get the
corresponding morphism $\sigma(\al)\colon \sigma\xi\to \sigma \nu$. 
Below, we write $\Isom^{\mf X}_T(\xi,\nu)$ for all isomorphisms from $\xi$ to $\nu$ 
inducing the identity on the scheme $T$ and we 
write $\Aut^{\mf X}_T(\xi)$ for all isomorphisms of $\xi$ over $T$ (we drop the 
superscript $\mf X$ if no ambiguity can arise).

We describe the objects and morphisms of 
the groupoid 
$\mathfrak I^\sigma_{\mathfrak X}$ over a scheme $T$.
By the definition of fiber product of categories, the objects are pairs $(\ga, \xi)$, where $\xi$ is an object of $\mf X$ over $T$ 
and $\ga$ is in $\Isom_T(\sigma \xi,\xi)$. 
The isomorphisms of the groupoid are given by 
$(\ga, \xi) 
\to (\al \ga \al^{-1}, \al\xi)$ for $\al\in \Aut_T(\xi)$.

For quotient stacks $\mf X=[U/G]$, \emph{i.e.} under Assumption \ref{specific_assumption}, we can provide a quotient stack presentation of $\mathfrak I^\sigma_{\mathfrak X}$. Here, $\sigma$ and $G$ are contained 
within a group acting properly on $U$ and satisfying $\sigma G\sigma^{-1}=G$. 
Then, the inertia scheme $I_{ G\sigma}(U)$ is the base change of $G\sigma\times U\to U\times U$ via the diagonal $U\to U\times U$.
We have 
\[ \mathfrak I^\sigma_\mf X=[I_{ G\sigma}(U)/G], \]
where $G$ operates as before:
by conjugation
on the first factor and by multiplication on the left on the second factor $\al\cdot(g\sigma, x) =(\al g\sigma\al^{-1}, \al x)$ (it is easy to see 
that conjugation by $\al\in G$ maps $G\sigma$ to itself as a 
consequence of $\sigma G\sigma^{-1}=G$).

\begin{rem}
There is a natural, representable morphism from the 
stack $\mf I^\sigma_{\mf X}$ to the inertia 
stack of $[\mf X/\sigma]$
\[\mf I^\sigma_{\mf X}\longrightarrow \mf I_{[\mf X/\sigma]}.\]
Indeed, recall that $[\mf X/\sigma]$ is the stack associated to the prestack whose objects 
are objects of $\mf X$ and whose morphisms $\al\colon \xi\to \nu$ are pairs $[\sigma^i,\fie]$
with $\fie\colon \sigma^i\xi \to \nu$ (see \cite[Prop.~2.6]{Romagny}). 
For any object $\xi$ of $\mf X$ over $T$ we have a natural isomorphism 
within the category $[\mf X/\sigma]$
\[\sigma=[\sigma, \id_{\sigma\xi}]\in \Isom^{[\mf X/\sigma]}_T(\xi, \sigma\xi)\]
and, by composition, 
a functor associating to the object $(\ga,\xi)$ of $\mf I^\sigma_{\mf X}$ over $T$,
with $\ga$ belonging to $\Isom^{\mf X}_T(\sigma\xi,\xi)$,
the object $(\ga\sigma,\xi)$, where 
\[\ga \sigma\in \Aut^{[\mf X/\sigma]}_T(\xi).\]
The functor lands in the substack of objects of the form 
$(\ga\sigma, \xi)$ with automorphisms $\al\in \Aut^{\mf X}_T(\xi)<\Aut^{[\mf X/\sigma]}_T(\xi)$ acting as described above: $\al\cdot(g\sigma, x) =(\al g\sigma\al^{-1}, \al x)$.
\end{rem}

\subsection{The age function}\label{subsect:age}
If $\mf X$ is smooth, the age function is a non-negative, locally constant function on 
the inertia stack
\[ \age \colon \mf I^\sigma_{\mf X}\to \QQ. \]
We can briefly introduce the age function as follows: 
to each geometric point of $\mf I^\sigma_{\mf X}$, given by 
$(g \in \Isom(\sigma x,x),x\in \mf X)$, we attach a finite-order 
representation $g\sigma$ operating
on the tangent space of $[\mf X/\sigma]$ at $x$. 
We write $g\sigma$ as $(\alpha_1,\ldots,\alpha_n)$ and compute 
$\age(g,x)=\mathrm{age}({g\sigma})$ as in \eqref{eq:Reidnotn}.

The actual definition of $\age$ in terms of 
objects of $\mf I_{\mf X}^\sigma$ over a connected scheme $T$
can be given as in \cite{AGV} through the above morphism 
$\mf I^\sigma_{\mf X}\to \mf I_{[\mf X/\sigma]}.$
To each pair
$(\gamma,\xi)\in \mf I_{\mf X}^\sigma(T)$, we attach $(\gamma\sigma, \xi)\in 
\mf I_{[\mf X/\sigma]}^\sigma(T)$ as above. As pointed out in \cite{AGV}, 
in the presence of distinguished identifications $\pmmu_r\to \ZZ/r$,
the inertia stack decomposes into the disjoint union over $r\in \NN^\times$ of 
cyclotomic inertia 
stacks $\mf I_{\pmmu_r}$ formed by objects $(\tau,\xi)$, where $\xi$ is an object 
of $[\mf X/\sigma]$ over $T$ and $\tau$ is 
an injective morphism from the trivial $\pmmu_r$-group scheme 
$(\pmmu_r)_T$ over $T$ into the automorphism group scheme of $\xi$ over $T$
\[(\pmmu_r)_T\longrightarrow \Aut^{[\mf X/\sigma]}_T(\xi).\]
In this way, the tangent bundle of $[\mf X/\sigma]$ pulls back to a $\pmmu_r$-linearized
bundle over $T$. The age function is the age of the $\pmmu_r$-representation in the sense of Section \ref{sect:notnfiniteorderaction}. Since $T$ is connected 
and the age function is locally constant, we obtain in this way 
the constant function $\age$ on $T$.

For quotient stacks $\mf X=[U/G]$, we can lift 
the function $\age$ to 
a $G$-invariant function on 
the $G\sigma$-inertia $U$-scheme $I_{G\sigma}(U)$ as follows.
The tangent bundle $T_{U}$ 
pulls back to $I_{G\sigma}(U)$ 
via the projection on $U$.
At each geometric point $(g\sigma,x)$ 
of $I_{G\sigma}(U)$, the group element $g\sigma$ 
operates on the $n$-dimensional 
 fiber of $T_U$ at $x$ as 
a finite-order representation $(\alpha_1,\ldots,\alpha_n)$ and 
$\mathrm{age}({g\sigma})$ yields 
a locally constant $G$-invariant function $\age$. 

\subsection{Orbifold cohomology} Orbifold cohomology classes are ordinary 
cohomology classes of the inertia stack bigraded after a shift.

If we ignore the grading,
the Chen--Ruan cohomology of the Deligne--Mumford stack $\mf X$ is simply the cohomology of the inertia stack $\mf I_{\mf X}$, which, in the case of $\mf X=[U/G]$, coincides
with the cohomology of $I_{G}(U)/G$ (the coarse moduli space) over the 
complex numbers. In our setup, $I_G(U)/G$ and $I_{G\sigma}(U)/G$ 
are quasi-smooth and 
admit Hodge decompositions:
\[H^n([I_{G}(U)/G];\CC)=
\bigoplus_{p+q=n}H^{p,q}(I_{G}(U)/G;\CC)\]
and \[ H^n([I_{G\sigma}(U)/G];\CC)=
\bigoplus_{p+q=n}H^{p,q}(I_{G\sigma}(U)/G;\CC).\]

Starting from this 
decomposition of weight $n$, for any $r\in \QQ$, we can produce a new 
decomposition of weight $n-2r$ via a shift analogous to the the Tate twist 
(see for instance \cite[\S4.3]{Yasuda})
\begin{equation}
\label{eq:Tate}H(r)^{p,q}=H^{p+r,p+r}.
\end{equation}
We provide the following definition, which extends the ordinary 
definition of Chen--Ruan cohomology by introducing the $\sigma$-inertia stack.

\begin{defn}\label{defn:sigmainertiacohomology}
The $\sigma$-orbifold 
cohomology is given by 
\begin{align*}
H_{\sigma}^*(\mathfrak X;\CC):&=H^*(\mathfrak I^\sigma_{\mf X};\CC)
(-\age)\end{align*}
\end{defn}

\begin{rem}
Above, the cohomology of the inertia stack 
is shifted by the locally constant function 
$\age$, which 
transforms classes of bidegree $(p,q)$ 
into classes of bidegree $(p+\age,q+\age)$.
Note the abuse of notation: $\age$ is not constant in 
general, but, since it is locally constant, the Tate shift operates independently on each cohomology group arising from each connected component. A precise notation should read 
\[ H^{p,q}(\_;\CC)(\mathfrak a)=\bigoplus_{r\in \QQ_{\ge 0}}H^{p,q}(\mathfrak a^{-1}(r);\CC)(-r). \]
\end{rem}

\begin{rem}[Chen--Ruan cohomology]
The definition of $\sigma$-shifted orbifold cohomology coincides with Chen--Ruan cohomology for $\sigma=\id$; we have
\[ H_{\id}^*(\mathfrak X;\CC)=H_{\CR}^*(\mathfrak X;\CC). \]
\end{rem}

\begin{rem}[quotient stacks]\label{rem:orbcohomquotstack} Under Assumption \ref{specific_assumption}, for $\mf X=[U/G]$, we have 
\begin{align*}
H_{\sigma}^*([U/G];\CC)&=
H^*([I_{G\sigma}(U)/G];\CC)(-\age)=H^*(I_{G\sigma}(U)/G;\CC)(-\age),\end{align*}
where, in the last equality, we have replaced the quotient stack by the coarse quotient,
because the expression only involves ordinary cohomology over complex coefficients.

In the specific cases mentioned in \S\ref{subsect:applications},
we work with abelian groups $G$ that can be expressed the 
extension of a finite group $K$ by a group $T\subseteq G$
$$0\to T\to G\to K\to 0.$$
Then, we can regard the quotient 
$I_{G\sigma}(U)/G$ as a finite group quotient $(I_{G\sigma}(U)/T)/K$.
Since $K$ is finite, we can write 
\[ H_{\sigma}^*([U/G];\CC)=\left(H^*(I_{G\sigma}(U)/T;\CC)(-\age)\right)^K.\]

As mentioned in \S\ref{subsect:applications},
in our setup $I_S(U)$ is the union of finitely many fixed loci in $U$
$$I_S(U)=\bigsqcup_{g\in S}U_g,$$
where $U_g=\{x\in U\mid g\cdot x=x\}.$
Therefore, in these cases, we can write 
\[ H_{\sigma}^*([U/G];\CC)=\bigoplus_{s\in G\sigma} H^*(U_s/T;\CC)(-\age)^K,\]
where the restriction to $K$-invariant loci can be considered summand by summand
because $K$ acts on $(g,x)$ by fixing the first coordinate. 
\end{rem}

%
%

\subsection{Orbifold cohomology groups attached to the data
\texorpdfstring{$W$}{W}, \texorpdfstring{$\Aut_W$}{AutW}, and \texorpdfstring{$\sigma$}{sigma}}{\label{sect:orbdata}}
Let $W$ be an invertible (non-degenerate) polynomial in the sense of 
Section \ref{sect:invertibleW}.
For any subgroup $G$ of $\Aut_W$ containing $j_W$,
we consider the quotient stack 
\[ \Sigma_{W,G}=[Z(W)/G\GG_m], \]
where $Z(W)$ is the zero locus of $W$ in $\CC^n-\{0\}$ and 
$G\GG_m$ is the group of diagonal matrices of 
the form $\diag(\al_1\la^{w_1},\dots,\al_n\la^{w_n})$
with $\diag(\al_1,\dots,\al_n)\in G$ and 
$\la\in \GG_m$. 

\begin{rem}[weighted hypersurfaces]
We notice that the smooth stack 
$Z_{\PP(\pmb w)}(W)$ is a special case of the above construction (see Remark \ref{rem:zeroloci})
\[ Z_{\PP(\pmb w)}(W)=\Sigma_{W,\langle j_W\rangle}. \]
 
Moreover, $\Sigma_{W,G}$ may be regarded as 
the quotient of $Z_{\PP(\pmb w)}(W)$ by 
 \[ G_0=G/\langle j_W\rangle; \] 
according to \cite[Rem.~2.4]{Romagny}, we have 
\[ [Z_{\PP(\pmb w)}(W)/G_0]=[Z(W)/G\GG_m]. \] 

\end{rem}

\begin{rem}[group actions on stacks]\label{rem:Dimca}
The group $G\GG_m$ is an abelian extension of $G_0$ by $\GG_m$. By Remark \ref{rem:orbcohomquotstack}, we have 
\[H_{\CR}^*(\Sigma_{W,G};\CC)=
\bigoplus_{s\in G} H^*(Z(W_s)/\GG_m;\CC)(-\age)^G,\]
where $W_s$ is the restriction of 
$W$ to $\CC^n_s$, $Z(W_s)$ is the zero locus 
within $\CC^n_s\setminus \pmb 0$ and $\GG_m$ operates with weights $\pmb w_s=(w_j\mid j\in F_s)$.
By definition,
the coarse quotient $Z(W_s)/\GG_m$ is quasi-smooth (see Remark \ref{rem:quasismooth}). 
 Notice that we consider 
$G$-invariant classes, which is equivalent to consider $G_0$-invariant classes 
because $j_W$ operates trivially on 
$Z(W_s)/\GG_m$.
\end{rem}

\begin{rem}\label{rem:agerounddown}
The function $\age $ takes the constant value $\mathrm{age}(\gamma)\in \QQ$ on each term $Z(W_s)$, because the age function on such hypersurface is related to that of the weighted projective space $\PP(\pmb w_s)$, where it lies by the following equation.
We have 
\begin{equation}\label{rem:agecorrection}
	\mathrm{age}(s\colon T_{Z(W)}\to T_{Z(W)}) = \mathrm{age}(s \colon T_{\CC^n}\to T_{\CC^n})-\mathrm{age}(s\colon N\to N),
\end{equation}
where $N$ is the normal bundle of $Z(W)$ within $\CC^n$.
Here, all ages are considered at a point $x\in Z(W_s)$.
Now, if we assume that the defining polynomial $W$ has degree $d$, then 
$s=(\al_1\la^{w_1},\dots,\al_n\la^{w_n})\in S\GG_m$ acts as $\la^d$ on the normal line.
This shows in particular that 
the value of $\age$ on $Z(W_s)$
is constant.
The reader may refer to \cite[Lemma~22]{ChRu} for an explicit proof. 
\end{rem}

\begin{rem}[the involution $\sigma$]
In this paper, we work with invertible polynomials of the form \[ W=x_0^2+f(x_1,\dots,x_n).\] 
Then $\Aut_W$ contains a distinguished symmetry $\sigma$ changing the sign of $x_0$ and fixing the remaining coordinates; with notation \eqref{eq:Reidnotn} we write
\[ \textstyle\sigma=\left(\frac12,0,\dots,0\right). \] 
Then, Remark \ref{rem:orbcohomquotstack} reads 
\begin{equation}\label{eq:sigmaorbifoldmadeexplicit}H_{\sigma}^*(\Sigma_{W,G};\CC)=
\bigoplus_{s\in G\sigma} H^*(Z(W_s)/\GG_m;\CC)(-\age)^G\end{equation}
\end{rem}

\begin{rem}\label{rem:integergrading} 
Assume $G\subseteq \SL_W\subset \SL_n(\CC)$; then, 
$\Sigma_{W,G}$ is Gorenstein (the stabilizers locally operate 
with determinant $1$). We have the following consequences. The group 
$H_{\CR}^{p,q}(Z_{W,G};\CC)$ is non-zero only if 
$p,q\in \ZZ$. Similarly, 
$H_{\sigma}^{p,q}(Z_{W,G};\CC)$ is non-zero only if 
$p,q\in \det \sigma+\ZZ=\frac12+\ZZ.$
\end{rem}

\subsection{Orbifold cohomology groups and ordinary cohomology}
The main application of the standard orbifold 
cohomology groups is the crepant resolution theorem proven by Yasuda in \cite{Yasuda}.

\begin{thm}[Yasuda, \cite{Yasuda, YasudaWild}]\label{thm:crthm}
If $\mathfrak X$ is a smooth, Gorenstein, Deligne--Mumford stack 
whose coarse moduli space $X$ admits a crepant resolution 
$\wt X\to X$, then $H^*_{\CR}(\mathfrak X;\CC)$ and $H^*(\wt X;\CC)$ are isomorphic as bigraded vector spaces. \end{thm}

For sake of clarity we provide a precise proof of the statement in the form needed here. 
We are grateful to Yasuda for guiding us through his results and recent improvements.  

\medskip 

\noindent \textit{Proof of Theorem \ref{thm:crthm}.}
It suffices to show that the resolution $\wt X$ 
and the stack $\mf X$ are $K$-equivalent:
there exists a smooth and proper Deligne--Mumford stack $\mf Z$ with birational 
morphisms $\mf Z\to \mf X$ and $\mf Z\to \wt X$ 
with $\omega_{\mf Z/\mf X}\cong \omega_{\mf Z/\wt X}$. 
Then the claim follows from  
\cite[Cor.~4.8]{Yasuda}: we 
deduce the above bigraded isomorphism 
by identifying the two sides  with $H^*_{\CR}(\mathfrak Z;\CC)$. 

The argument is simple but not immediate. In order to see this, let us consider the 
fibred product $\mf Z'=\mf X\times _X \wt X$
and the projections on the two factors
\[
\xymatrix@R=.99pc{
&\mf Z'\ar[dr]\ar[dl] & \\
{\mf{X}}\ar[dr]_{\mf{p}} & & {\wt X}\ar[dl]^{\wt p}\\
& X &
}
\]
whose relative dualizing bundles 
are isomorphic. 
Then we consider the associated reduced stack $\mf Z''$
and the resolution $\mf Z \to \mf Z''$ whose existence is shown
in the second paragraph of section 4.5 of Yasuda's paper \cite{Yasuda} 
(this is essentially due to Villamayor papers \cite{Vil1} and \cite{Vil2} showing  the existence of 
resolutions compatible with smooth  morphisms).
It should be also mentioned that, in his recent generalization \cite{YasudaWild}, Yasuda 
allows us to get 
the desired claim directly via the reduction and normalization $\mf Z$ of $\mf Z'$, without
passing through any resolution. This happens because his new statements allows us to
extend the definition 
of $H^*_{\CR}(\ ;\CC)$
to singular or wild (in positive characteristic) 
Deligne--Mumford stacks. This allows us to
deduce the desired bigraded isomorphism between $H^*_{\CR}(\mathfrak X;\CC)$ and $H^*(\wt X;\CC)$
by identifying them directly with $H^*_{\CR}(\mathfrak Z';\CC)$, see \cite[Thm.~16.2]{YasudaWild}.\qed 

\medskip 
Let us now consider our setup. 
The existence of a crepant resolution is guaranteed
in dimension $2$ and $3$. In dimension $2$, the 
crepant resolution is canonical and coincides with 
the minimal resolution. In the present paper we also have an involution 
$$\sigma\colon \mf X\to \mf X$$
acting by change of sign on the determinant of the tangent bundle 
locally at any fixed point. 
Since the coarse space $X$ is 
the final object with respect to morphisms to algebraic spaces we 
have an involution of $X$, still denoted by $\sigma$. 
Since $\mf X$ is Gorenstein, 
$\omega_{\mf X}$ descends to $X$. Since $\wt X$ is the minimal resolution, 
$\sigma$ lifts to $\wt X$. Locally 
at a fixed point of $X$, $\sigma$ acts by change of sign 
on the determinant of the tangent space at each fixed point.
We can rephrase more explicitly this condition, at each fixed point $p\in\wt X$, $\sigma$ operates linearly on $T=T_p(\wt X)$: we can decompose $T$ into a $c$-dimensional subspace of anti-invariant vectors and an invariant subspace of codimension $c$
\begin{equation}\label{eq:decompTangent}
T=T^+\oplus T^-\cong \CC^{\dim(\wt X)-c}\oplus \CC^c.\end{equation} 
Since $\det(\sigma \colon T\to T)=-1$, as we assume that the dimension is $2$, we can write $\sigma$ locally as $\diag(-1,1)$
(\emph{i.e.} $c=1$). 
Furthermore, as $\sigma$ operates on $\mf X$ and $\wt X$ compatibly, 
we get 
\[
\xymatrix@R=.99pc{
&[\mf Z/\sigma]\ar[dr]\ar[dl] & \\
{[\mf{X}/\sigma]}\ar[dr]_{\mf{p}} & & {[\wt X/\sigma]}\ar[dl]^{\wt p}\\
& [X/\sigma] &
}
\]
The conditions of $K$-equivalence mentioned above are satisfied. Then, the above statement 
by Yasuda yields 
\[ H^{p,q}_{\CR}([\mf X/\sigma];\CC)\cong H^{p,q}_{\CR}([\wt X/\sigma];\CC),\]
for $p,q\in \frac12\ZZ$ (we refer to \cite[Cor.~4.8]{Yasuda}).
By unravelling the definition of $\sigma$-orbifolded cohomology 
and by restricting to the $(\frac12,\frac12)+\ZZ\times \ZZ$-graded part 
of the above isomorphism, we get 
an identification \begin{equation}\label{eq:sigmasigma}
H^{p,q}_{\sigma}(\mf X;\CC)\cong H^{p,q}_{\sigma}(\wt X;\CC)
\end{equation}
between the $\sigma$-orbifold cohomology 
of $\mf X$ and of $\wt X$.
Finally, since $\sigma$ acts as $\diag(-1,1)$ locally at each fixed point of $\wt X$, 
we get an isomorphism between the $(1/2)$-shifted
$\sigma$-orbifold cohomology of $\mf X$ 
and the ordinary cohomology of 
the fixed locus $\wt X_\sigma$ in $\wt X$. 
We get the following isomorphism 
of bigraded vector spaces.
\begin{pro}\label{pro:fixedlocusstated}
Let $\sigma\colon \mf X\to \mf X$
be an involution of a $2$-dimensional stack
 operating by a change of sign on the canonical bundle 
locally at each fixed point of $\mf X$. Then, we have the bigraded isomorphism
\begin{equation*}H^*_{\sigma}(\mathfrak X;\CC)\textstyle(\frac12)
\cong H^*(\wt X_\sigma;\CC),\end{equation*}
where $\wt X_\sigma$  is the fixed locus of the involution $\sigma$
induced on the minimal resolution $\wt X$ of the coarse space $X$ of $\mf X$.
\end{pro} 

Since we are focusing on the $2$-dimensional case and 
we are only concerned with the special case $\mf X=\Sigma_{W,G}$
with an involution of the form 
$\sigma=(\frac12,0,0,0)$, we can provide an explicit realisation of the  
isomorphism of Proposition \ref{pro:fixedlocusstated} without passing through Yasuda's statements
and by relying on the ADE-classification of singularities. We proceed as follows.

Since we assume that the dimension is $2$, and since the stabilizers are abelian 
and satisfy the Gorenstein condition (as a consequence of $\omega_{\mf X}\cong \Ocal_{\mf X}$), we get a local presentation of 
$\mf X$ of the form $[\CC^2/\pmmu_r]$ 
with $\zeta\in \pmmu_r$ acting as $\zeta\cdot(x,y)=(\zeta x,\zeta^{-1} y)$. This 
happens since we are in the (abelian) A-case of the  ADE-classification of 
Du Val singularities. 
The involution $\sigma$ not only operates directly on $Z(W)$ in $\CC^4\setminus 0$ as mentioned in \S\ref{subsect:applications}. Actually we can 
further argue that, locally at each
$\sigma$-fixed point,
its local picture is the same as that of an involution 
$\sigma\colon [\CC^2/\pmmu_r]\to [\CC^2/\pmmu_r]$ 
induced by a linear map $\CC^2\to \CC^2$ commuting with the action of $\pmmu_r$. 

Indeed, one can study $[(\CC^4\setminus \pmb 0)/\GG_m]$ explicitly via the 
open loci $x_j\neq 0$ 
 $$[\CC^3\times (\CC_{x_j}\setminus \pmb 0)/\GG_m]=[\CC^3/\pmmu_{w_j}]$$ 
with $\xi_{w_j}=\exp(2\pi i/w_j)$ acting diagonally as
$\diag(\xi_{w_j}^{w_i})_{i\neq j}$. 
The above identity follows from 
writing $\CC^3\times (\CC_{x_j}\setminus \pmb 0)$ as $\CC^3\times Y$
with $Y=\CC_{x_j}\setminus \pmb 0$ and 
$\GG_m$ operating on $\CC^3\times Y$ and on $Y$. 
Since the action is transitive on $Y$,
the stack $[\CC^3\times Y/\GG_m]$ equals $[\CC^3/S]$, where $S$ is the stabilizer of the action of 
$\GG_m$ on $Y$. 
We apply the same method to $G\GG_m$ and we present the open set 
$(x_j\neq 0)$ within 
$[(\CC^4\setminus \pmb 0)/G_m]$ as $[\CC^3/H]$, where $H$ is the finite stabilizer of the 
action of $G\GG_m$ on $\CC_{x_j}\setminus \pmb 0$. 
Then, we restrict the hypersurface $\Sigma_{W,G}$ and obtain a smooth hypersurface 
$\Sigma_{W,G}\times_{\PP(\pmb w)} [\CC^3/H]$ within $[\CC^3/H]$. 
Locally at a point $p$ of the hypersurface $Z(W)$, the hypersurface
can be identified to $[\CC^2/K]$, with the stabilizer $K$ of $H$ at $p$
acting linearly as a finite subgroup of $\GL(2;\CC)$. 
Indeed, since $\Sigma_{G,W}$ has trivial canonical bundle, $K$ 
belongs to $\SL(2;\CC)$. Since $K$ is abelian, after a change of basis, it is necessarily cyclic and operates 
as $(\frac1r,\frac{r-1}r)$. 
Finally, the same argument applies to $G(\sigma)\GG_m$ and yields a local 
presentation of the 
stack $[\Sigma_{W,G}/\sigma]$ as $[\CC^2/K']$ with $K'\subset \GL(2;\CC)$ containing a 
cyclic subgroup $K\subseteq \SL(2;\CC)$ of index $2$ so that $[\CC^2/K]$ is a local presentation of $\Sigma_{W,G}$.

We can now provide the statement which applies to $\Sigma_{W,G}$ 
and affirms the existence of an explicit 
isomorphism. 
This is an improvement on the more general
approach via Yasuda's result: indeed we get an explicit identification rather than
an identity between the dimensions of the vector spaces involved. 
We restate Proposition \ref{pro:fixedlocusstated}
in the setup where we can provide such  isomorphism. 
\begin{pro}\label{pro:crthm}
Consider a $2$-dimensional, smooth, proper, Gorenstein,
Deligne--Mumford stack $\mathfrak X$ 
with abelian stabilizers at each point. 
Let $\sigma\colon \mf X\to \mf X$ be an involution 
acting by change of sign on the determinant of the tangent space 
locally at any fixed point $x\in \mf X$; more precisely, we assume 
that the local picture of $\sigma$ and $\mf X$ at $x$ 
is the local picture of $\iota$ and $[\CC^2/G]$ at the origin, 
with $G\subset \SL(2;\CC)$ and $\iota\in \GL(2,\CC)$ of determinant $-1$ satisfying $\iota G\iota^{-1}=G$.

We consider the minimal resolution $\wt X$ 
of the coarse moduli space $X$, the induced involution, still denoted $\sigma$, and the fixed space $\wt X_\sigma$ in $\wt X$. 
Then, we have an explicit 
identification of bigraded vector spaces 
\begin{equation}\label{eq:fixedlocus}
H^*_{\sigma}(\mathfrak X;\CC)\textstyle(\frac12)\cong H^*(\wt X_\sigma;\CC).\end{equation}
\end{pro}
\begin{proof}The local picture at each point $p\in \mf X$ with non-trivial stabilizer is given by 
$$\mf U=[\Spec \CC[x,y]/\pmmu_{d_p}], \qquad \text{with $d_p\in \NN_{>1}$ 
and $\zeta\in \pmmu_{d_p}$ operating as $\diag (\zeta,\zeta^{-1})$}.$$
The stack $\mf X$ may be regarded as the union of 
the representable locus $\mf X^{\mathrm{rep}}$ and of 
the stacks $[\Spec \CC[x,y]/\pmmu_{d_p}]$ identified along 
$(\Spec \CC[x,y]\setminus \pmb 0)/\pmmu_{d_p}$
for every $p$ in the set $P$ of points with non-trivial stabilizer.
 \begin{lem}
 Let $\mf U =[\Spec \CC[x,y]/\pmmu_{d}]$ and let $\sigma$ be an involution of $\Spec \CC[x,y]$ descending to an involution of $\mf U$ and
operating by change of sign on $dx\wedge dy$. Then $\sigma$ 
is isomorphic to exactly one of the following cases:
\begin{enumerate}
\item[(a)] involutions $\sigma(x,y)=(-x,y)$ 
fixing each branch of the node $(xy=0)$ 
and acting trivially on at least one 
branch;
\item[(b)] involutions $\sigma(x,y)=(\la y,\la^{-1}x)$ (with $\la\neq 0$)
switching the two branches; 
\item[(c)] involutions $\sigma(x,y)=(\xi_{2d} x,\xi_{2d}^{d-1}y)$
mapping each branch to itself and fixing only one point.
\end{enumerate} 
We have an explicit isomorphism
$H^*_{\sigma}(\mathfrak U;\CC)\textstyle(\frac12)\cong H^*(\wt U_\sigma;\CC),$
where $\wt U$ is a crepant resolution of the coarse space $U$ of $\mf U$.
\end{lem}
\begin{proof}
Before proving the statement, let us 
point out that the three cases are disjoint. In case (c), for $d=2$, we have  $\sigma(x,y)=(\xi_4 x,\xi_4 y)$ which has determinant $-1$ as required. For $d=3$, we have 
$\sigma=\diag(\xi_6 , \xi_6^2)$ and
 \[\sigma \diag(\xi_3,\xi_3^{-1})=\diag(-1,1);\] we see that this is actually a case (a) example. More generally, it suffices in case (c) to consider
only $d \in 2 \ZZ$, as otherwise 
the natural transformation $\diag(\xi_d,\xi_d^{-1})^{(d-1)/2}$ identifies $\sigma$ to case (a).

We describe the $\sigma$-inertia stack 
$\mf I_{\mf U}^\sigma$ explicitly in each case.

\medskip 

\noindent Case (a). For $\sigma(x,y)=(-x,y)$, we get 
\begin{align}
\mf I^\sigma_{[\CC^2/\pmmu_{d}]}&=[I_{\pmmu_d\sigma}(\Spec \CC[x,y])/\pmmu_d]\nonumber\\
\nonumber&=\bigg [\big\{(\zeta\sigma,(x,y))\mid \zeta\in \pmmu_d, \ (-\zeta x,\zeta^{-1}y)=(x,y)\big\}\ /\ \pmmu_d\bigg ]\\ 
&=\begin{cases}
[\Spec \CC[y]/\pmmu_d]_0\ \sqcup \ \bigsqcup_{j=1}^{\frac{d}2 -1} (B\pmmu_{d})_j 
\ \sqcup\ 
[\Spec \CC[y]/\pmmu_d]_{\frac{d}2}\ \sqcup\ \bigsqcup_{j=\frac{d}2+1}^{d-1} (B\pmmu_{d})_j& \text{$d\in 2\ZZ$,}\\
[\Spec \CC[y]/\pmmu_d]_0\ \sqcup \ \bigsqcup_{j=1}^{d-1} (B\pmmu_{d})_j
& \text{else.}
\label{eq:onP}
\end{cases}\end{align}
The cohomology is shifted by 
the age function which is 
constantly equal to $\frac12$ on the
one-dimensional components.
The zero-dimensional components $(B\pmmu_d)_j$ corresponding to the index $j$ are 
\[ (B\pmmu_d)_j=[(\xi_d^j\sigma,\pmb 0)/\pmmu_d].\]
We notice that $\xi_d^{j}$ operates on $\CC^2$ 
as $\diag(-\xi_d^j, \xi_d^{-j})$ in 
\eqref{eq:onP}; 
the age is 
$1+\frac12$ for $i<d/2$ and $\frac12$ for $i>d/2$.
The exceptional divisor of the 
crepant resolution $\wt U$ of $U=\{uv=t^d\}$ consists of $d-1$ curves 
$E_1,\dots, E_{d-1}$ with $E_i$ intersecting $E_{i-1}$ and $E_{i+1}$ for $i\neq 1, d-1$ 
and $E_1$ (resp.~$E_{d-1}$) intersecting $E_2$ (resp.~$E_{d-2}$).
We refer to the proper transform of $(u=0)$ as $E_0$ and to the proper 
transform of $(v=0)$ as $E_{d}$; these are the coarse images of the branches $(x=0)$ and $y=0$ in $\mf U$. 
Since $(x,y)\in \mf U$ maps to $(x^d,y^d)\in U$ the branches $E_0$ and $E_d$ are fixed 
if $d$ is even, whereas, for odd $d$, only $E_0$ is fixed. Since $\wt U$ is the minimal resolution 
$\sigma$ lifts to $\wt U$ to a unique involution locally acting as $\diag(-1,1)$; we 
deduce that $E_{2i}$ are fixed with $0<i<d$. 
The claim follows since
the cohomology of a projective line 
is $V=\CC\oplus \CC(-1)$ and 
we have 
\[H^*_\sigma(\mf U;\CC)(\textstyle\frac12)=H^*(\mf I^\sigma_{\mf U};\CC){\textstyle(-\mf a+\frac12)}= \CC^2\ \oplus 
V^{\oplus\lfloor(d-1)/2\rfloor}\cong \bigoplus_{\substack{0\le j \le d \\{j\in 2\ZZ}}} H^*(E_j;\CC)=H^*(\wt U_\sigma;\CC).\]

\medskip 

\noindent Case (b). For $\sigma(x,y)=(\la y,\la^{-1}x)$, we get 
\begin{align}
\mf I^\sigma_{[\CC^2/\pmmu_{d}]}&=[I_{\pmmu_d\sigma}(\Spec \CC[x,y])/\pmmu_d]\nonumber\\
\nonumber&=\bigg [\big\{(\zeta\sigma,(x,y))\mid \zeta\in \pmmu_d, \ (\zeta\la y,\zeta^{-1}\la^{-1}x)=(x,y)\big\}\ /\ \pmmu_d\bigg ]\\
&=\begin{cases}
[\Spec (\CC[x,y]/(x=\la y))/\pmmu_2]_0\ \sqcup 
[\Spec (\CC[x,y]/(x=-\la y))/\pmmu_2]_1& \text{$d\in 2\ZZ$,}\\
\Spec (\CC[x,y]/(x=\la y))_0
& \text{else.}\end{cases}
\label{eq:onP2}
\end{align}
The labels $0$ and $1$ in the even case and the label $0$ in the odd case 
indicate that there are two or one conjugacy classes represented 
by $\sigma$ and (for even $d$) by $\diag(\xi_d,\xi_d^{-1})\sigma$; this happens 
because conjugating $\sigma$ by $\diag(\xi_d,\xi_d^{-1})$ yields 
$\diag(\xi_d^2,\xi_d^{-2})\sigma$.
The age is constantly $1/2$; therefore the claim boils down to 
the identity between the cohomology of the two lines (resp. one line) above 
and the cohomology of $\wt U_\sigma$. We notice that $\sigma$ exchanges $E_i$ and $E_{d-i}$ for $i=0,\dots, n$. It fixes two smooth points in $E_{{d}/2}$ or the node $E_{({d-1})/2}\cap E_{({d+1})/2}$, where the proper transform of the fixed locus in $\mf U\setminus \pmb 0$ meet the exceptional divisor. The fixed locus reduces to the proper 
transform of $(\mf U\setminus \pmb 0)_\sigma$; \emph{i.e.} two lines (resp.~one line) if $d$ is even (resp.~odd). The claim follows.

\medskip

\noindent Case (c). For $\sigma(x,y)=(\xi_{2d} x,\xi_{2d}^{d-1} y)$, we get 
\begin{align}
\nonumber\mf I^\sigma_{[\CC^2/\pmmu_{d}]}&=
\bigg [\big\{(\zeta\sigma,(x,y))\mid \zeta\in \pmmu_d, \ (\xi_{2d}\zeta x,\xi_{2d}^{d-1}\zeta^{-1}y)=(x,y)\big\}\ /\ \pmmu_d\bigg ]\\&=\bigsqcup_{j=0}^{d-1} (B\pmmu_d)_j.
\label{eq:onP3}
\end{align}
The zero-dimensional components $(B\pmmu_d)_j$ corresponding to the index $j$ are 
$[(\xi_d^j\sigma,\pmb 0)/\pmmu_d]$, where $\xi_d^{j}$ operates on $\CC^2$ 
as $\diag(\xi_{2d}^{2j+1}, \xi_{2d}^{3d-2j+1})$;
the age is 
$1+\frac12$ for $d/2\le j<d$ and $\frac12$ for $j<d/2$.

Using the same notation as in Case (a), the automorphism $\sigma$ fixes every second exceptional curve $E_i$, and neither of $E_0$ or $E_d$, so that it fixes precisely $E_j$ for $0<j<d$ and $j \in 2\ZZ+1$.
We have 
\[H^*_\sigma(\mf U;\CC)(\textstyle\frac12)=H^*(\mf I^\sigma_{\mf U};\CC){\textstyle(-\mf a+\frac12)}=
V^{\oplus d/2}\cong \bigoplus_{\substack{0< j <d \\{j\in 2\ZZ+1}}} H^*(E_j;\CC)=H^*(\wt U_\sigma;\CC).\]

\end{proof}

The $\sigma$-inertia stack of the 
representable stack $\mf X^{\mathrm{rep}}$ is simply a smooth curve
\begin{equation}\label{eq:offP}
\mf I_{\mf X^{\mathrm{rep}}}^\sigma=(\mf X^{\mathrm{rep}})_\sigma=(X\setminus P)_\sigma=X_\sigma\setminus P=\wt X_\sigma\setminus P\end{equation}
whose proper transform within $\wt X$ coincides with the coarse space of the $1$-dimensional part of the $\sigma$-inertia stack. This identification and the 
above lemma complete the proof. The $\frac12$-shift is 
due to the constant value $\frac12$ of the 
age on the $1$-dimensional components of the inertia stack. 
\end{proof}
The argument via Yasuda's theorem generalizes in any dimension
under the conditions that 
\begin{enumerate}
\item a crepant resolution $\wt X$ of the coarse space $X$
exists, 
\item 
that the induced involution $\sigma\colon X\to X$ lifts to 
$\wt X$, and
\item that the fixed locus is a divisor in $\wt X$. \end{enumerate}
Condition $(1)$ holds in dimension $3$ (and often fails in higher dimension). 
Condition $(2)$ needs to be checked explicitly; we point out 
that in the cases $\mf X=\Sigma_{W,G}$ the situation is further 
simplified by the fact that the involution operates on 
the weighted projective space as $(\frac12,0,\dots,0)$; however this does not 
allow to express $\sigma$ as $(\frac12,0,\dots,0)$ in general.
Condition $(3)$ is only used to deduce 
\eqref{eq:fixedlocus} from \eqref{eq:sigmasigma}. 

Without 
condition $(3)$, the involution $\sigma$ locally acts as 
$-\id_{\CC^c}\oplus \id_{\CC^{n-c}}$ at each fixed point of $\wt X$ (see \eqref{eq:decompTangent}). 
Therefore, we can decompose 
$\wt X_\sigma$ into the disjoint union of smooth open and closed 
subvarieties of odd codimension $\wt X_\sigma^1, \dots ,\wt X_\sigma^{1+2\lfloor d/2\rfloor}$
$$\wt X_\sigma=\bigsqcup_{\substack{ c\in 1+2\ZZ\\ c\le \dim(X)}}\wt X_\sigma^c,$$
where $\wt X_\sigma^c$ is a smooth subvariety of $\wt X$ of codimension $c$. The same argument used to prove Proposition \ref{pro:fixedlocusstated} implies the following statement where we assume conditions $(1)$ and $(2)$ above and modify the claim as follows.
\begin{pro}\label{pro:alldim}
Let $\mathfrak X$ be a smooth, proper, Gorenstein,
Deligne--Mumford stack
with trivial stabilizer on the 
generic point. 
Let $\sigma\colon \mf X\to \mf X$ be an involution
acting by change of sign on the determinant of the tangent space 
locally at any fixed point. 
 If $X$ admits a crepant resolution $\wt X\to X$, and the involution induced 
by $\sigma$ on $X$ admits a lift to $\wt X$, then we have an isomorphism of bigraded vector spaces 
$$H^*_{\sigma}(\mathfrak X;\CC)\textstyle(-\frac12)\cong
\bigoplus_{c} H^*(\wt X_\sigma^c;\CC)(-c).$$
\end{pro}

\section{Landau--Ginzburg models} \label{sect:LG-models}
Orbifold cohomology yields a bigraded vector space attached to an invertible polynomial~$W$, a
 subgroup $G\ni j_W$ of $\Aut_W$, and --- in the 
 generalized version presented here --- an automorphism.
In this section, we define another bigraded vector space associated to the non-degenerate polynomial $W$ and its symmetries: the Landau--Ginzburg model. We then discuss results from the literature relating these two state spaces, as well as mirror symmetry for LG models.

\subsection{LG state space}\label{sect:statespaces}
For any degree-$d$ quasi-homogeneous 
non-degenerate polynomial $W$ in the 
variables $x_1,\dots, x_n$ of weights $w_1,\dots,w_n$
(regardless of any Calabi--Yau condition 
on the sum of the weights or even any 
invertibility condition on the defining 
matrix), we consider the (full) 
state space of the Landau--Ginzburg model $W\colon \CC^n\to \CC$
\[
	\mathcal H^{*,*}(W):=\bigoplus_{g\in \Aut_W} \Jac(W_g),
\]
where each summand is given by the Jacobi ring 
\[
\Jac(W_g)=\bigg[\CC[x_j\mid j\in F_g]/(\partial_j W_g\mid j\in F_g)
\bigg]\ \textstyle{\bigwedge_{j\in F_g} dx_j},
\]
and where $\partial_j$ stands for $\partial/\partial x_j$ (and we refer to \S\ref{sect:notnfiniteorderaction} for $F_g$ and $W_g$).
 The following definition of degree of a monomial allows us
to make the Jacobi ring 
into a graded algebra as follows.
We will use the notation $[g,\phi]$ for an element 
in the image of $\Jac(W_g)\hookrightarrow \Hcal(W)$, where $\phi = \prod_{j\in F_g} x_j^{m_j} \bigwedge_{j\in F_g} dx_j$ is a monomial term of degree 
\[ 
	\deg(\phi):=\frac1d\sum_j(m_j+1){w_j}.
\]
The bigrading $(p,q)$ of an element $[g,\phi] \in \Jac(W_g)$ is given by
\begin{equation}\label{eq:bigrading}
	(p,q):=\left(\#F_g-\deg(\phi)+\mrage(g), \deg(\phi)+\mrage(g)\right),
\end{equation}
Since $\deg(\partial_j W_g \bigwedge_{j\in F_g} dx_j)=1$ for all $g$ and for all $j$, 
the Jacobi ideal is a homogeneous ideal 
and the grading descends
to the Jacobi ring.

In this paper, we regard $\Hcal(W)$ as a bigraded vector space and
we never use its ring structure (\emph{e.g.} \cite{FJR1}).
In the notation of \eqref{eq:Reidnotn}, any diagonal symmetry $\alpha = (a_1,\ldots,a_n) \in \Aut_W$ 
acts on $[g,\phi]$ as
$\al^*[g,\phi]=\chi_\al [g,\phi]$, where 
$\chi_\al $
is
 the multiplication by $\exp(2\pi i t_{(a_1,\ldots,a_n)}) $ with $t_{(a_1,\ldots,a_n)}$ 
 given by 
\[
	t_{(a_1,\ldots,a_n)}=\sum_j(m_j+1) a_j\in \QQ/\ZZ \qquad \text{(and $\phi = \prod_{j\in F_g} x_j^{m_j} \bigwedge_{j\in F_g} dx_j$)}.
\]
Note that the action of the grading element $j_W$ is actually given by the character~$\deg$ mod $\ZZ$.

The space $\Hcal(W)$ is also referred to as ``unprojected'' state space in the literature
(see~\cite{Krawitz}). This is because,
when studied with respect to a group action,
it can be projected onto several 
subspaces of invariant elements playing a role in the theory of Landau--Ginzburg models.
We treat group actions and invariant 
subspaces systematically here. 

Given a subgroup $K$ of $\Aut_W$, and a subset $S\subseteq\Aut_W$, we define 
\begin{equation}\label{eq:unproj}
	\mathcal H^{*,*}_S(W)^K := \bigoplus_{g\in S} \Jac(W_g)^K.
\end{equation}

\begin{rem}[FJRW state space] In the special case where $S=K$, we recover the definition of the Fan--Jarvis--Ruan--Witten state space of a Landau--Ginzburg model $W\colon[\CC^r/S]\to \CC$. 
For any invertible polynomial $W$ and any subgroup $K$ of $\Aut_W$, via a Tate twist by $q=\sum_j q_j=\sum_j w_j/d$, we
have \[\mathcal H_{\FJR}(W,K)=\Hcal_K(W)^K{\textstyle(q)}.\]
If $W$ is a Calabi--Yau invertible polynomial, then
the charges add up to $1$ and we have
\[
	\mathcal H_{\FJR}^{p,q}(W,K)=\Hcal^{p+1,q+1} _K(W)^K.
\]
\end{rem}

\subsection{Dual polynomials and groups}\label{sect:dualpolys}
We now consider an invertible Landau--Ginzburg potential 
as in \eqref{eq:invertibleW}. We still avoid imposing 
any Calabi--Yau condition on the sum of weights and the degree.

Let $W$ be a non-degenerate Landau--Ginzburg potential in
$n$ variables, whose exponent matrix is the invertible matrix $M=(m_{i,j}).$ The dual polynomial, denoted $W^\vee$, is defined as the Landau--Ginzburg potential whose exponent matrix is given by $M^{\mathrm T}$.
Following \eqref{eq:Reidnotn}, 
the columns of $M^{-1}=(m^{i,j})$ are generators of 
$\Aut_W$ and the rows of $M^{-1}$ are generators of 
$\Aut_{W^\vee}$. 

\begin{rem}\label{rem:canonicalident}
We recall that, by Remark \ref{rem:lattices}, the Cartier duality $H^*=\Hom(H;\GG_m)$ induces a canonical isomorphism (see \cite{Berglund-Henningson,GZE_Saito, Gusein-Zade-Ebeling})
\[
	({\Aut}_W)^* \cong \Aut_{W^\vee}.
\]
For any subgroup $G$ of $\Aut_W$, the Berglund--Hübsch dual group to $G$ is \[ G^\vee=\ker ({i}^*\colon \Aut_{W^\vee}\to {G}^*). \] 
The duality reverses the inclusions and transforms into each other two distinguished groups: $J_W:=\langle j_W\rangle$ into $SL_{W^\vee}:=\Aut_W^\vee\cap \SL_n(\CC)$ and $SL_{W}$ into $J_{W^\vee}$.
\end{rem}

The following theorem is proven in various versions in 
\cite{Kre, Krawitz, BorisovBH}. 
We provide the statement of 
\cite{Kre} and \cite{Krawitz}. To each monomial
$x_1^{a_1}\cdots x_r^{a_r}$, we attach a diagonal symmetry 
as follows:
\[
	\gamma\colon x_1^{a_1}\cdots x_r^{a_r}\mapsto 
\textstyle{ \prod_{j\in F_g} 
(m^{j,1},\dots, m^{j,n})^{a_j}},
\]
where $m^{i,j}$ are the coefficients of the inverse 
of the exponent matrix of $W$ (refer to notation~\eqref{eq:Reidnotn}). 
The right hand side lies in $\Aut_{W^\vee}$, because the lines 
of the inverse matrix $M$ span $\Aut_{W^\vee}$.
With a slight abuse of notation identifying 
the form $\prod_{j\in F_g}x_j^{a_j-1}\bigwedge_{j\in F_g} d x_j$
to the monomial $\prod_{j\in F_g} x_j^{a_j}$,
we can apply $\gamma$ to each summand of $\Hcal(W)$:
\[
	\gamma\colon \Jac(W_g)\to \Aut_{W^\vee}.
\]

\begin{rem}\label{rem:Krawitzduality}
In particular, $\gamma$ provides an 
equivalent interpretation of the dual group 
$G^\vee$ attached to any 
subgroup $ G$ of $\Aut_W$. We have 
\[
	\ker ({i}^*\colon \Aut_{W^\vee}\to {G}^*) = \gamma(\{G\text{-invariant monomials}\}),
\]
where the right hand side is actually Krawitz's 
original formalization of 
the standard Berglund--Hübsch duality.\end{rem}

\begin{thm}[Unprojected Landau--Ginzburg mirror symmetry, \cite{Krawitz}, \cite{BorisovBH}]\label{thm:Kra}
We have an isomorphism 
\begin{align*}
\Gamma\colon \Hcal^{p,q}(W)&\to \Hcal^{n-p,q}(W^\vee).
\end{align*}
The isomorphism attaches to each element of the form 
\[
	[h,\phi=\textstyle{\prod_{j\in F_h}x_j^{a_j-1}
\bigwedge_{j\in F_h} x_j}],
\]
a unique element of the form 
$[\gamma(\phi), T]$ with $T\in \Jac(W^\vee_{\gamma(\phi)})\cap \gamma^{-1}\{h\}$.\qed
\end{thm}

\begin{rem}\label{rem:mirrorswitch} Let $W=x_0^2+f(x_1,\dots,x_n)$. Then 
$\Aut_W=(\sigma)\times \Aut_f$, where $\sigma=(\frac12,0,\dots,0)$. Notice that $[g,\phi]$ is $\sigma$-invariant if and only if $g\not\in \Aut_f$. Furthermore, $g\in \Aut_f\subset \Aut_W$ implies $\gamma(\phi)\not \in \Aut_f$. We conclude that $\Gamma$ exchanges $\sigma$-invariant and $\sigma$-anti-invariant terms. 
\end{rem}

\begin{cor}\label{cor:Kra}
For any $S,K\subseteq \Aut_W$, 
consider the Berglund--Hübsch dual groups $H^\vee, K^\vee\subseteq \Aut_{W^\vee}$;
then, $\Gamma$ yields an isomorphism 
\[
	\Gamma\colon \Hcal^{p,q}_{S}(W)^{K}\to 
\Hcal_{K^\vee }^{n-p,q}(W^\vee )^{S^\vee}.
\]
\end{cor}
\begin{proof}
We only need to show that the image of 
$\Hcal_{S}(W)^{K}$ via $\Gamma$ is contained in 
$\Hcal_{K^\vee}(W^\vee)^{S^\vee}$. Then, the same claim holds in the opposite sense and
we conclude by Theorem \ref{thm:Kra}.

Given $[h,\phi]\in \Hcal_{S}(W)^{K}$, we need to prove that
the image $[\gamma(\phi),T]$ lies in 
$\Hcal_{K^\vee}(W^\vee)^{S^\vee}$.
First,
$\gamma(\phi)$ lies in $K^\vee$, 
because, by Remark \ref{rem:Krawitzduality} 
we have $\ga(\Jac(W_h))\subseteq K^\vee$.
Second, the form $T$ is $S^\vee$-invariant. Indeed, 
this amounts to proving that $T$ is 
invariant with respect to any symmetry of 
the form $\gamma(M)$ for any $S$-invariant 
monomial $M$. The last claim 
is equivalent to showing that $\gamma^\vee(T)$ fixes any $S$-invariant 
monomial 
$M$; this is the case because we have $\gamma^\vee(T)=h$ and $h\in S$. 
\end{proof}

\subsection{LG/CY correspondence} \label{sec:LGCY}
We slightly generalize
the Landau--Ginzburg/Calabi--Yau correspondence of \cite{ChRu} to 
$\sigma$-orbifold cohomology and 
to the state spaces $\Hcal_S(W)^K$ above.
Although in this paper we only apply this theorem to invertible polynomials,
 we do not need any invertibility condition on the polynomial here.
On the other hand, it is essential to 
require that all groups of symmetries involved 
in the statement contain $j_W$. 
\begin{thm}[Chiodo--Ruan \cite{ChRu}]\label{thm:LGCY}
Let $W$ be any non-degenerate polynomial
of weights $w_1,\dots, w_r$ and degree $d=w_1+\dots+w_r$ (Calabi--Yau condition). 
Let $G$ be a group of diagonal symmetries
containing $j_W$. Consider any automorphism $\varepsilon\in \Aut_W$ and the induced 
automorphism $\varepsilon\colon \Sigma_{W,G}\to \Sigma_{W,G}$. 
Then, for any $p$ and $q\in \QQ$, we have 
\[
	H_{\varepsilon}^{p,q}(\Sigma_{W,G};\CC)
	\cong \Hcal^{p,q}_{G\varepsilon}(W)^G(1),
\]
where 
the isomorphism is compatible with any finite-order 
diagonal symmetry 
acting on $\CC^r$ and preserving $W$.
\qed\end{thm}
\begin{proof} 
We argue as in \cite{ChRu}, where first the case where $G=J_W$ is shown: an 
explicit bidegree preserving isomorphism 
\[
	H_{\CR}(Z_{\PP(\pmb w)}(W);\CC)\cong \Hcal_{J_W}(W)^{J_W}(1)
\]
is given. 
We spell out the definitions of the two sides of the equation above, by writing 
\[
\bigoplus_{s\in \GG_m} H^*(Z(W_s)/\GG_m;\CC)(-\age)\cong
\bigoplus_{g\in J_W} \Jac(W_g)^{J_W}.
\]
We notice that the statement proven in \cite{ChRu}
is actually a bidegree-preserving automorphism 
\[
\bigoplus_{s\in \gamma \GG_m} H^*(Z(W_s)/\GG_m;\CC)(-\age)
 \cong\bigoplus_{g\in \gamma J_W} \Jac(W_g)^{J_W}.
\]
for any $\gamma\in \Aut_W$ (the proof proceeds $J_W$-coset by $J_W$-coset). 
In this way 
we have 
\[
H_{\sigma}^{p,q}(Z_{\PP(\pmb w)}(W);\CC) \cong \Hcal_{J_W\sigma}(W)^{J_W}(1).
\]
Finally, by summing over all cosets of $J_W$ in $G$ and by taking invariants with respect to $G$ 
on both sides, we get the desired claim.
\end{proof}

\begin{rem}
Note that for $\epsilon=\id$, the theorem above identifies Chen--Ruan cohomology and FJRW state spaces. 
Then, by combining it with 
with the statement of
Berglund--Hübsch 
mirror symmetry for Calabi--Yau invertible polynomials, we get 
the following claim from \cite{ChRu}. For any 
group $G$ containing $j_W$ and included in $\SL_W$,
we have 
\begin{multline*} H_{\CR}^{p,q}(\Sigma_{W,G};\CC)\cong\Hcal_{\FJR}^{p,q}(W, G)\cong 
\Hcal^{p+1,q+1}_{G}(W)^{G}\\\cong 
H^{n-p-1,q+1}_{G^\vee}(W^\vee)^{G^\vee}=
H^{n-2-p,q}_{G^\vee}(W^\vee)^{G^\vee}(1)\\\cong 
\Hcal_{\FJR}^{n-2-p,q}(W^\vee, G^\vee)
\cong H_{\CR}^{n-2-p,q}(\Sigma_{W^\vee,G^\vee};\CC).\end{multline*}
\end{rem}

\section{Mirror symmetry with an involution} \label{sect:involutions}

The classical Borcea--Voisin construction involves K3 surfaces with an involution. The generalization that we consider here is higher dimensional Calabi--Yau hypersurfaces within weighted projective spaces equipped with an involution. 

\subsection{Polynomials with an involution}
We will consider invertible Calabi--Yau polynomials of the form
\[ W(x_0,x_1,x_2,\dots,x_n)=x_0^2+f(x_1,\dots,x_n), \]
for some invertible polynomial $f$.
Consider $\sigma_W=(\frac12,0,\dots,0) \in \Aut_W$; we will usually write $\sigma$ omitting the subscript $W$ when 
no ambiguity may arise. Note that we can view $\Aut_f$ as a subgroup of $\Aut_W;$ in particular, $j_f \in \Aut_W$ and $\SL_f$ is contained in $\SL_W\subset\Aut_W$.

We also consider a group $H$ such that 
\[j_W \in H \subseteq \SL_W.\] 
Consider the surjective map 
$H\to \ZZ/2$, defined as the restriction to 
the $x_0$-line; the exact sequence \[0\to \ol H\to H\to \ZZ/2\to 0\]
splits; we have $(\sigma)\times \ol H=H$ with $\ol H\subseteq \Aut_f$. We have 
$H=\ol H[j_W]$ and $j_W=\sigma j_f$. 
The condition $j_W \in H \subset \SL_W$ implies 
$j_f^2 \in \ol H \subset \SL_f.$ 

Before studying the mirror dual of $W$ and $H$, let us consider the Landau--Ginzburg state space 
$\Hcal_{H\sigma}(W)^H$ in the light of the Landau--Ginzburg/Calabi--Yau theorem \ref{thm:LGCY}, which
matches 
$\Hcal_{H\sigma}(W)^H$ with $H_\sigma^*(\Sigma_{W,H};\CC)$ when
 $W$ is of Calabi--Yau type. 
Therefore, each element of 
$\Hcal_{H\sigma}(W)^H$ is $\sigma$-invariant.
In the special case where $\Sigma_{W,H}$ is $2$-dimensional, this can be 
observed scheme-theoretically: by Proposition \ref{pro:crthm}, 
we have an isomorphism between 
$\Hcal_{H\sigma}(W)^H$
and the $\sigma$-fixed locus of the resolution of 
$\Sigma_{W,H}$, whose cohomology is obviously fixed by $\sigma$. 
However, the following proposition shows that the $\sigma$ invariance of $\Hcal_{H\sigma}(W)^H$ can be seen purely on the LG side and therefore is true even without the CY condition. 
\begin{pro}\label{pro:noantiinvoninv} For any $W=x_0^2+f(x_1,\dots, x_n)$, $j_W\in H\subseteq \SL_W$ and $\sigma=(\frac12,0,\dots,0)$,
the involution $\sigma$ acts trivially on $\Hcal_{H\sigma}(W)^H$.
\end{pro}
\begin{proof}
We prove that the $\sigma$-anti-invariant part
$\Hcal_{H\sigma}(W)^{H}_{-}$ of 
$\Hcal_{H\sigma}(W)^H$ vanishes. 
Let us first consider the $\ol H$-invariant part 
\[
	\Hcal_{H\sigma}(W)^{\ol H}_{-}=\Hcal_{\ol H j_f }(W)^{\ol H} \cong \Hcal_{\ol H j_f }(f)^{\ol H}.
\]
The first identity is due to the fact that 
a $\sigma$-anti-invariant element is necessarily of the form $[h\sigma,dx_0\wedge \phi]$ with 
$h\sigma \in \ol Hj_f$). The second isomorphism maps
$[gj_f\in \ol Hj_f , dx_0 \wedge \phi]$ 
to $[gj_f ,\phi]$.
Now let us assume that $[g j_f \in \ol Hj_f, dx_0 \wedge \phi]\in 
\Hcal_{\ol H j_f }(W)^{\ol H}$ is a non-zero 
 $j_W$-invariant element (\emph{i.e.} it lies in 
 $\Hcal_{H\sigma}(W)^H$). We get a contradiction. 

First, write $g\in \ol H$ as $(p_1,\dots,p_n)\in \Aut_f.$ Then $g j_f=(w_1/d+p_1,\dots,w_n/d+p_n)$ and 
the set $I=F_{g j_f}$ of the indices of fixed coordinates is $I=\{i | w_i/d+p_i \in \ZZ\}$. 
We can write $\phi$ as 
\[\prod_{i \in I} x_i^{a_i-1} \bigwedge_i dx_i.\]
The form $\phi$ is $\ol H$ invariant so
$\sum_{i \in I} p_i a_i \in \ZZ,$
which implies
$\sum_{i \in I} -a_i w_i/d \in \ZZ.$
But this contradicts the assumption that 
$dx_0 \wedge f$ is $j_W$-invariant, which implies 
$	\sum_{i \in I} a_i w_i/d \in \frac{1}{2}+\ZZ.$
\end{proof}

\subsection{Mirror duality of CY-polynomials with involution}
The dual polynomial $W^\vee$ is also 
of the form $W^\vee=x_1^2+f^\vee$ and possesses a symmetry $\sigma_{W^\vee}$; by abuse of notation we
refer to $\sigma_{W^\vee}$ as $\sigma$. 
As shown above, the group $H^\vee$ is included in $\SL_{W^\vee}$ and contains $j_{W^\vee}$. We have 
$j_{W^\vee}^2\in \ol{H^\vee}\subset \SL_{W^\vee}$.

\begin{pro}\label{pro:propoertiesofstar} 
For any $H$ satisfying $j_W\in H\subseteq \SL_W$, we have $j_{W}^2\in \ol H\subseteq \SL_{W}$ and 
$j_{W^\vee}^2\in \ol{ H^\vee} \subseteq \SL_{W^\vee}$.\end{pro}
\begin{proof}
Under the canonical identification of 
Remark \ref{rem:canonicalident}, 
requiring that a character of $\Aut_{W^\vee}$
 vanishes on $j_W$ is equivalent to imposing 
the condition $\det=1$ within the group of 
diagonal symmetries of $W^\vee$. 
Therefore, if $i$ is the inclusion $H\into \Aut_W$ 
and $\bar \imath$ is the inclusion $\ol H\into \Aut_W$, we have 
\[
	\ker {i^*}=\ker (\bar {\imath}^*) 
	\cap \SL_{W^\vee}.
\]
The condition $j_{W^\vee}^2 \in \ol{ H^\vee }$ is satisfied because 
$\ker i^*=H^\vee$ contains $j_{W^\vee}$ and
has index two in $\ker (\ol \imath^*)$, because $\ol H$ has index two in $H$.\end{proof}

\begin{pro}\label{pro:BHdualities}
In the above setup, the 
inclusion-reversion operation 
$\vee$
exchanges the following two diagrams 
\[
	\xymatrix@C=.99pc{
 & \ol H\ar[dl]_{i_1}\ar[dr]^{i_3}\ar[d] ^{i_2}&&& & & & \ol{H^\vee}[\sigma ,j_{f^\vee}] & \\
\ol H[\sigma]\ar[dr]_{i_4} & H=\ol H[\sigma j_f] \ar[d]^{i_5} & \ol H[j_f]\ar[dl]^{i_6} &&\overset{\vee}\rightsquigarrow& & 
 \ol{H^\vee}[j_{f^\vee}]\ar[ur]^{i_1^\vee} & H^\vee
 =\ol{H^\vee}[\sigma j_{f^\vee}] \ar[u]_{i_2^\vee} & \ol{H^\vee} [\sigma ]\ar[ul]_{i_3^\vee} \\
 & \ol H[\sigma, j_f] && && & & \ol{H^\vee} \ar[ul]^{i_4^\vee}\ar[ur]_{i_6^\vee}\ar[u]_{i_5^\vee} &, \\
}
\]
where all the arrows are injective homomorphisms (on both sides, the groups of the form $\ol H[\sigma,j_f]$ may be 
regarded as $H[\sigma]=H[j_f]$).
\end{pro}
\begin{proof}
Indeed, $\sigma j_f$ and $\sigma j_{f^\vee}$ are 
the grading elements of $x_1^2+f$ and $x_1^2+f^\vee$. Therefore
$H[\sigma j_f]$ 
is dual to $\ol{ H^\vee} [\sigma j_{f^\vee}]$. This explains the middle lines of the above transformations (the inclusions 
are reversed due to Proposition \ref{pro:propoertiesofstar}).
Finally, $\ol H[\sigma]$ is a direct product of $(\sigma)$ and $\ol H$ (automorphism 
groups of summands involving disjoint sets of variables). 
Therefore, 
Berglund--Hübsch duality $\vee$ yields the direct product of the dual of $\sigma$ 
within $\Aut(x_0^2)$, which is trivial, 
with the direct product of the Berglund--Hübsch 
dual of $\ol H$ in $\Aut_f$, which is $\ol{H^\vee}[ j_{f^\vee}]$.
\end{proof}

\subsection{The geometric mirror symmetry theorem}\label{subsect:geomMS}
The proof of the main mirror symmetry statement follows naturally from
the previous setup, the LG/CY correspondence and mirror symmetry on the LG side. 

Consider a pair $(W=x_0^2+f,H)$ as above, 
with $W$ of CY-type, non-degenerate and invertible 
and $H\subseteq \SL_W$ and containing $j_W$. 
We realize that the 
state space $\Hcal_{H[\sigma]}(W)^H$ attached to the three data $W, H$, and $\sigma$ decomposes into all 
the relevant cohomological data. Indeed, we have
\begin{equation}\label{eq:LGCYdouble}\Hcal^{p,q}_{H[\sigma]}(W)^H=\Hcal^{p,q}_{H}(W)^H
\oplus \Hcal^{p,q}_{H\sigma}(W)^H=H^{p,q}_{\CR}(\Sigma_{W,G};\CC)\oplus H^{p,q}_\sigma(\Sigma_{W,G};\CC),\end{equation}
where the LG/CY correspondence has been used on both factors in the form of 
Theorem \ref{thm:LGCY}.
By Remark \ref{rem:integergrading}, the first summand
is $\ZZ\times \ZZ$-graded and the second is 
$(\frac12,\frac12)+\ZZ\times \ZZ$-graded.
By the Landau--Ginzburg mirror symmetry theorem \ref{thm:Kra}, we have 
\begin{equation}\label{eq:mirrorincomplete}\Hcal^{p,q}_{H[\sigma]}(W)^{H}\cong 
\Hcal^{n-p,q}_{H^\vee}(W^\vee)^{(H[\sigma])^\vee}
=\Hcal^{n-p,q}_{H^\vee}(W^\vee)^{\ol{H^\vee}},\end{equation}
where in the second equality Proposition \ref{pro:BHdualities} yields $(H[\sigma])^\vee =\ol{H^\vee}$.
We study the last term after decomposing it into 
its $j_W$-invariant $(\ )_{j_W,+}$ part and its $j_W$-anti-invariant $(\ )_{j_W,-}$ part. 
In general, we have 
\begin{equation}\label{eq:beforelemma}\Hcal_{H}(W)^{\ol H}=\left (\Hcal_{H}(W)^{\ol H}\right)_{j_W,+} \oplus 
\left(\Hcal_{H}(W)^{\ol H}\right)_{j_W,-},\end{equation}
where the first summand is $\Hcal_{H}(W)^{H}$ because $\ol H$-invariance and $j_W$-invariance is equivalent to $H=\ol H[j_W]$-invariance. For the second summand we use the following result.
\begin{lem}\label{lem:strippingtwisting}
For any $j_W\in H\in \SL_W$, there is an explicit isomorphism 
which preserves the bidegree and 
exchange $\sigma$-invariant terms into $\sigma$-anti-invariant terms
\[
	\left(\Hcal_{H}(W)^{\ol H}\right)_{j_W,-}\cong\Hcal_{H\sigma}(W)^{H}.
\]
\end{lem}
\begin{proof}
The left-hand side is spanned by 
$j_W\text{-anti-invariant}$ terms of the two following forms 
\begin{align}\label{anti}
&\left[g, \prod_{j \in F_g\setminus \{0\}}x_j^{a_j-1} dx_0 \wedge \bigwedge_{j \in F_g} d x_j\right],\\
&\left[j_W g, \prod_{j \in F_{j_W g}} x_j^{a_j-1} \bigwedge_{j \in F_{j_W g}} d x_j \right],\label{inv}
\end{align}
where $g \in \ol H$.
Note that the spaces spanned elements of 
type \eqref{inv} is the $\sigma$-invariant
subspace and that the elements \eqref{anti}
span the $\sigma$-anti-invariant subspace.

The right hand side decomposes as follows, 
$$\Hcal_{H\sigma}(W)^H=\Hcal_{\ol H\sigma}(W)^H\oplus \Hcal_{\ol Hj_W \sigma}(W)^H.$$


The first summand 
$\Hcal_{\ol H\sigma }(W)^H$
is identified with the anti-invariant elements of $\Hcal_{H}(W)^{\ol H}_{j_W,-}$, given by elements of the form \eqref{anti}. The identification is given by
\[\left[g, \prod_{j \in F_g\setminus \{0\}}x_j^{a_j-1} dx_0 \wedge \bigwedge_{j \in F_g} d x_j\right]
	 \mapsto
	 \left[\sigma g, \prod_{j \in F_{\sigma g}} x_j^{a_j-1} \bigwedge_{j \in F_{\sigma g}} d x_j \right].
\]
Note that the new form is $j_W$-invariant; therefore we land in $\Hcal_{\ol H\sigma }(W)^H$.

The second summand of
$\Hcal_{\ol H\sigma }(W)^H$
is identified with the anti-invariant elements of $\Hcal_{H}(W)^{\ol H}_{j_W,-}$, given by elements of the form \eqref{inv}. The identification is defined by rewriting $j_W g, g \in \ol H$ in \eqref{inv} as
$\sigma \ol g$ for $\ol g=j_f g$; then we set
\[
	\left[\sigma \ol g, \prod_{j \in F_{\sigma g}} x_j^{a_j-1} \bigwedge_{j \in F_{\sigma g}} d x_j \right] \mapsto \left[\ol g, \prod_{j \in F_g\setminus \{0\}}x_j^{a_j-1} dx_0 \wedge \bigwedge_{j \in F_g} d x_j\right].
\]
We land in $\Hcal_{\ol Hj_W \sigma }(W)^H = \Hcal_{\ol H j_f}(W)^H$.

It is immediate to check that these map
 preserves the bidegree. As pointed out above, 
 they exchange $\sigma$-anti-invariant terms to $\sigma$-invariant terms.
\end{proof}

Then, we rewrite \eqref{eq:beforelemma} as 
\begin{equation}\label{eq:afterlemma}\Hcal_{H}(W)^{\ol H}=\Hcal_{H}(W)^{H} \oplus 
\Hcal_{H\sigma}(W)^{H},\end{equation}
Therefore, we can complete \eqref{eq:mirrorincomplete}
as follows
$$\Hcal^{p,q}_{H[\sigma]}(W)^{H}\cong 
\Hcal^{n-p,q}_{H^\vee}(W^\vee)^{(H[\sigma])^\vee}
=\Hcal^{n-p,q}_{H^\vee}(W^\vee)^{\ol{H^\vee}}
=\Hcal^{n-p,q}_{H^\vee[\sigma]}(W^\vee)^{{H^\vee}},
$$
In view of \eqref{eq:afterlemma}, this mirror map is
the direct sum of the two mirror maps
$$\Hcal^{p,q}_{H}(W)^{H}\cong \Hcal^{n-p,q}_{H^\vee}(W^\vee)^{H^\vee},\qquad \Hcal^{p,q}_{H\sigma}(W)^{H}\cong \Hcal^{n-p,q}_{H^\vee\sigma}(W^\vee)^{H^\vee};$$
indeed, the first isomorphism identifies 
$\ZZ\times \ZZ$-graded spaces and the second isomorphism identifies $(\frac12,\frac12)+\ZZ\times \ZZ$-graded spaces.

The first map is a direct application of 
Corollary~\ref{cor:Kra}; therefore it exchanges $\sigma$-invariant and $\sigma$-anti-invariant eigenspaces, 
see Remark \ref{rem:mirrorswitch}. 
The second map is the composite of two isomorphism which switch the $\sigma$-eigenspaces: the mirror symmetry theorem \ref{thm:Kra} and 
Lemma~\ref{lem:strippingtwisting}. Therefore, this isomorphism 
preserves the $\sigma$-invariant and $\sigma$-anti-invariant subspaces. 
We can now rephrase the two mirror theorems 
in light of the LG/CY correspondence: the first theorem concerns $H^*_{\CR}
(\Sigma_{W,H};\CC)$ and the second concerns
$H^*_{\sigma}
(\Sigma_{W,H};\CC)$. 
We summarize our results in the following statement where we used the notation $V^+, V^-$ 
to identify the $\sigma$-invariant subspace and the 
$\sigma$-anti-invariant subspace.
\begin{thm}\label{thm:semimirror}
Consider $(W=x_0^2+f,H)$, where $W$ is an invertible Calabi--Yau polynomial with involution $\sigma=(\frac12,0,\dots,0)$, and
$H$ satisfies $j_W\in H\subseteq \SL_W$.
Let $(W^\vee=x_0^2+f^\vee,H^\vee)$ 
be the dual pair. In all degrees $(p,q) \in \ZZ\times \ZZ$, we have 
\begin{enumerate}[(i)]
\item
$H^{p,q}_{\CR}
(\Sigma_{W,H};\CC)^\pm\cong
H^{n-1-p,q}_{\CR}(\Sigma_{W^\vee,H^\sigma};\CC)^\mp$;
\item
$H^{p,q}_\sigma(\Sigma_{W,H};\CC)(\frac12)\cong
H^{n-2-p,q}_\sigma(\Sigma_{W^\vee,H^\sigma};\CC)(\frac12)$.
\end{enumerate}
\end{thm}

The following corollary is a direct consequence of Proposition \ref{pro:alldim}.
\begin{cor}\label{cor:crepantsemi}
Let $\wt \Sigma$ and $\wt \Sigma^\vee$ be 
two crepant resolutions of the Gorenstein orbifolds 
$\Sigma_{W,H}$ and $\Sigma_{W^\vee,H^\vee}$
admitting an involution $\sigma$ lifting 
to $\wt \Sigma$ and $\wt \Sigma^\vee$ the involution $(\frac12,0,\dots,0)$. 
Then $\wt \Sigma_\sigma$ and $\wt \Sigma^\vee_\sigma$ are the disjoint union of 
varieties $\wt \Sigma(j)$ and $\wt \Sigma^\vee(j)$ 
of dimension $n-2-2j$ with $j\in \{0,1,\dots \lfloor\frac{n-2}2\rfloor\}$
and we have 
\begin{enumerate}[(i)]
\item
$H^{p,q}
(\wt\Sigma;\CC)^\pm\cong
H^{n-1-p,q}(\wt\Sigma^\vee;\CC)^\mp$;
\item
$\bigoplus_{j=0}^{\lfloor \frac{n}2\rfloor-1}H^{p-j,q-j}(\wt\Sigma_\sigma(j);\CC))\cong
\bigoplus_{j=0}^{\lfloor \frac{n}2\rfloor-1} H^{n-2-p-j,q-j}(\wt\Sigma^\vee_\sigma(j);\CC)$.
\end{enumerate}
\end{cor}

\begin{exa}
\label{exa:homog}
For any positive integer $n$, 
we consider 
\[
	W = x_0^2 + f, \quad \text{with} \quad f=x_1^{2n}+x_2^{2n}+\dots+x_n^{2n},
\]
together with $H = J_W$. Then, we consider the mirror; namely $W^\vee=W$ and $H^\vee=\SL_W$. 
The matching Hodge diamonds are presented for 
$n\le 5$ in Figures~\ref{fig:total-Hodge-fermat} and
\ref{fig:total-Hodge-mirror-fermat}.

The locus $\Sigma_W := \Sigma_{W,J_W} \subset \PP(n,1,\ldots,1)$ is represented by a smooth $(n-1)$-dimensional 
Calabi--Yau variety.
The Hodge diamond vanishes everywhere except for 
$h^{p,p}=1$, $0\le p\le n-1$, and 
the $n$-tuple $(h^{n-1,0},h^{n-2,1},\dots, h^{0,n-1})$. For $n=1,2,3,4,5$ we have listed the values of 
$(h^{n-1,0},h^{n-2,1},\dots, h^{0,n-1})$ in Figure \ref{fig:total-Hodge-fermat}.
Furthermore, since $\Sigma_{W}$ is representable, by definition, $H_{\sigma}(\Sigma_{W,G};\CC)$ is the cohomology of the $\sigma$-fixed locus 
$(\Sigma_{W,G})_\sigma$ up to an overall shift $(\frac12,\frac12)$. 
Notice that, in all these cases\footnote{In general, the $\sigma$-fixed locus is not 
reduced to $Z_{\PP(w_1,\dots,w_n)}(f)$; this happens because $\Sigma_{W}$ is a $\GG_m$-quotient stack (i.e. of the form $[U/\GG_m]$) and 
the $\sigma$-fixed locus is the fixed locus of $\sigma$ up to the $\GG_m$-action; see Example~\ref{exa:inhomog}.}, $(\Sigma_{W,G})_\sigma$ 
coincides with the smooth 
substack $Z_{\PP^{n-1}}(f)$. 
The Hodge diamond of the $\sigma$-cohomology vanishes except for 
$h_\sigma^{p,p}=1$, $0\le p\le n-2$, and 
the $n-1$-tuple $(h_\sigma^{n-2,0},h_\sigma^{n-3,1},\dots, h_\sigma^{0,n-2})$.
For $n=1$ there is no non-vanishing entry; for $n=2,3,4,5$ we have shown the explicit values in 
Figure~\ref{fig:total-Hodge-fermat}.

Figure \ref{fig:total-Hodge-fermat}
shows the five Hodge diamonds of 
$\Sigma_W$ for $n = 1,\ldots, 5$, inside 
which we have pictured 
the Hodge diamonds of $H^{*,*}_\sigma(W,G)$
 inscribed in square boxes whose coordinates belong to 
$(\frac12,\frac12)+\ZZ^2$. In this way we identify with a single 
rotation $(p,q)\mapsto (n-1-p,q)$ the two Hodge diamonds with the two mirror Hodge diamond for each $n$.
In view of a comparison with the mirror, we underline the Hodge numbers of the $\sigma$-anti-invariant part.

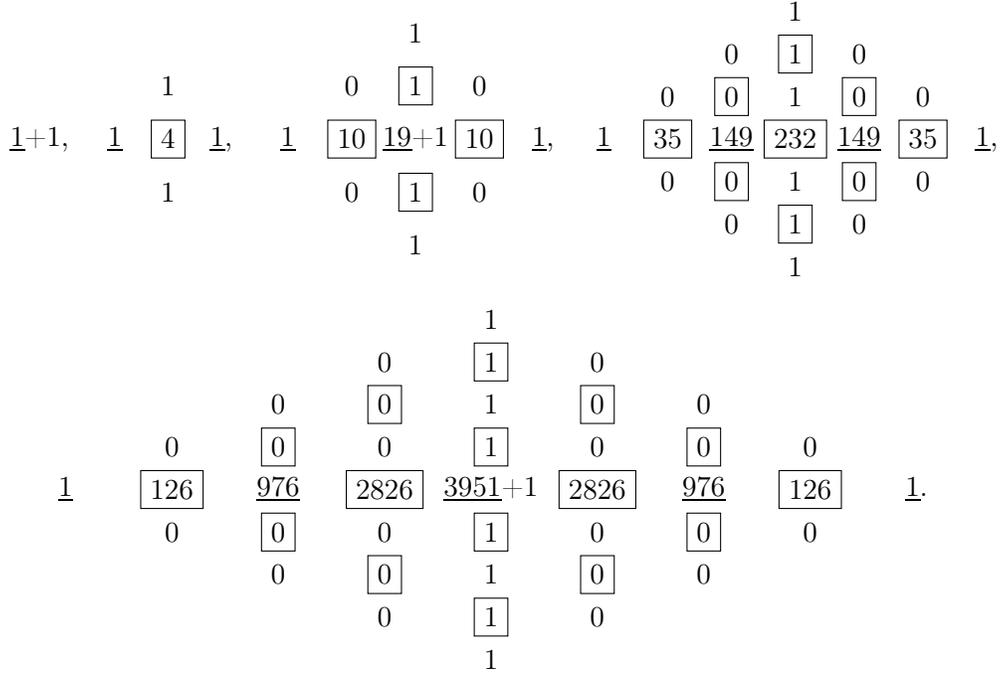
\begin{figure}[h!]
\begin{tikzpicture}

\begin{scope}[rotate=-135]
\node at (0,0) {\ul 1+1,};
\end{scope}

\begin{scope}[shift={($ (45:1) + (1,0) $)},rotate=-135]
\node at (0,0) {1};

\node at (1,0) {\ul 1};
\node at (1/2,1/2) [rectangle,draw] {4};
\node at (0,1) {\ul 1,};

\node at (1,1) {1};
\end{scope}

\begin{scope}
[xscale=1.2,shift={($ (45:2) + (11/4,0) $)},rotate=-135]
\node at (0,0) {1};

\node at (1,0) {0};
\node at (1/2,1/2) [rectangle,draw] {1};
\node at (0,1) {0};

\node at (2,0) {\ul 1};
\node at (3/2,1/2) [rectangle,draw] {10};
\node at (1,1) {\ul{19}+1};
\node at (1/2,3/2) [rectangle,draw] {10};
\node at (0,2) {\ul 1,};

\node at (2,1) {0};
\node at (3/2,3/2) [rectangle,draw] {1};
\node at (1,2) {0};

\node at (2,2) {1};
\end{scope}

\begin{scope}[xscale=1.2,yscale=0.8]
\begin{scope}[shift={($ (45:3) + (25/4,0) $)},rotate=-135]
\node at (0,0) {1};

\node at (1,0) {0};
\node at (1/2,1/2) [rectangle,draw] {1};
\node at (0,1) {0};

\node at (2,0) {0};
\node at (3/2,1/2) [rectangle,draw] {0};
\node at (1,1) {1};
\node at (1/2,3/2) [rectangle,draw] {0};
\node at (0,2) {0};

\node at (3,0) {\ul 1};
\node at (5/2,1/2) [rectangle,draw] {35};
\node at (2,1) {\ul{149}};
\node at (3/2,3/2) [rectangle,draw] {232};
\node at (1,2) {\ul{149}};
\node at (1/2,5/2) [rectangle,draw] {35};
\node at (0,3) {\ul 1,};

\node at (3,1) {0};
\node at (5/2,3/2) [rectangle,draw] {0};
\node at (2,2) {1};
\node at (3/2,5/2) [rectangle,draw] {0};
\node at (1,3) {0};

\node at (3,2) {0};
\node at (5/2,5/2) [rectangle,draw] {1};
\node at (2,3) {0};

\node at (3,3) {1};
\end{scope}
\end{scope}

\begin{scope}[xscale=2.0,yscale=0.8]
\begin{scope}[shift={(3,-3)},rotate=-135]
\node at (0,0) {1};

\node at (1,0) {0};
\node at (1/2,1/2) [rectangle,draw] {1};
\node at (0,1) {0};

\node at (2,0) {0};
\node at (3/2,1/2) [rectangle,draw] {0};
\node at (1,1) {1};
\node at (1/2,3/2) [rectangle,draw] {0};
\node at (0,2) {0};

\node at (3,0) {0};
\node at (5/2,1/2) [rectangle,draw] {0};
\node at (2,1) {0};
\node at (3/2,3/2) [rectangle,draw] {1};
\node at (1,2) {0};
\node at (1/2,5/2) [rectangle,draw] {0};
\node at (0,3) {0};

\node at (4,0) {\ul 1};
\node at (7/2,1/2) [rectangle,draw] {126};
\node at (3,1) {\ul{976}};
\node at (5/2,3/2) [rectangle,draw] {2826};
\node at (2,2) {\ul{3951}+1};
\node at (3/2,5/2) [rectangle,draw] {2826};
\node at (1,3) {\ul{976}};
\node at (1/2,7/2) [rectangle,draw] {126};
\node at (0,4) {\ul{1}.};

\node at (4,1) {0};
\node at (7/2,3/2) [rectangle,draw] {0};
\node at (3,2) {0};
\node at (5/2,5/2) [rectangle,draw] {1};
\node at (2,3) {0};
\node at (3/2,7/2) [rectangle,draw] {0};
\node at (1,4) {0};

\node at (4,2) {0};
\node at (7/2,5/2) [rectangle,draw] {0};
\node at (3,3) {1};
\node at (5/2,7/2) [rectangle,draw] {0};
\node at (2,4) {0};

\node at (4,3) {0};
\node at (7/2,7/2) [rectangle,draw] {1};
\node at (3,4) {0};

\node at (4,4) {1};
\end{scope}
\end{scope}

\end{tikzpicture}
\caption{\footnotesize The Hodge diamonds for $W = x_0^2 + x_1^{2n} + \cdots + x_n^{2n}$ and $H=J_W$, 
$n = 1,\ldots,5$.}
\label{fig:total-Hodge-fermat}
\end{figure}

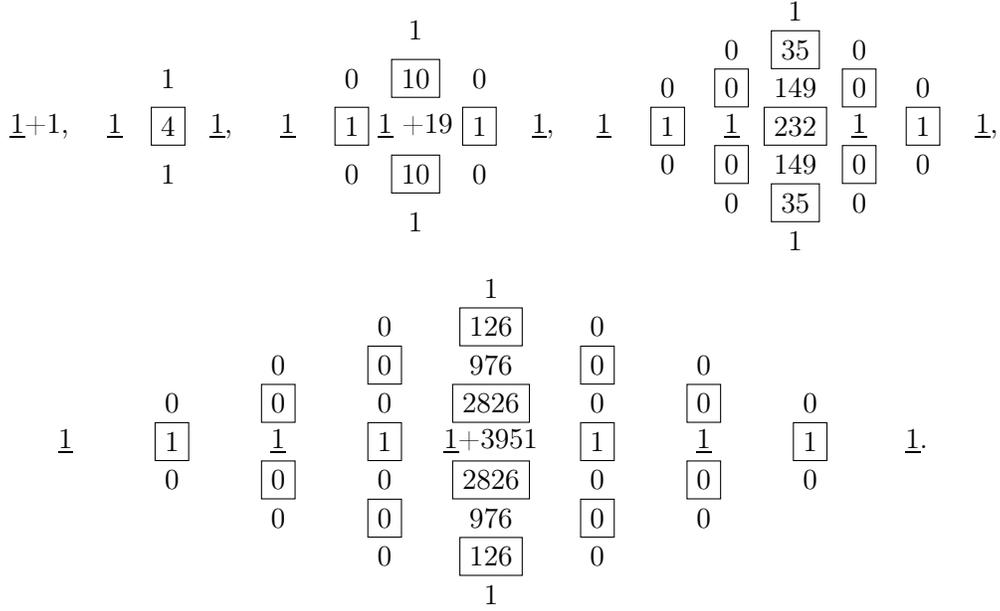
\begin{figure}[hh]
\begin{tikzpicture}[yscale=0.9]

\begin{scope}[rotate=-135]
\node at (0,0) {\ul 1+1,};
\end{scope}

\begin{scope}[shift={($ (45:1) + (1,0) $)},rotate=-135]
\node at (0,0) {1};

\node at (1,0) {\ul 1};
\node at (1/2,1/2) [rectangle,draw] {4};
\node at (0,1) {\ul 1,};

\node at (1,1) {1};
\end{scope}

\begin{scope}[xscale=1.2,shift={($ (45:2) + (11/4,0) $)},rotate=-135]
\node at (0,0) {1};

\node at (1,0) {0};
\node at (1/2,1/2) [rectangle,draw] {10};
\node at (0,1) {0};

\node at (2,0) {\ul 1};
\node at (3/2,1/2) [rectangle,draw] {1};
\node at (1,1) {\ul1 +{19}};
\node at (1/2,3/2) [rectangle,draw] {1};
\node at (0,2) {\ul 1,};

\node at (2,1) {0};
\node at (3/2,3/2) [rectangle,draw] {10};
\node at (1,2) {0};

\node at (2,2) {1};
\end{scope}

\begin{scope}[xscale=1.2,yscale=0.8]
\begin{scope}[shift={($ (45:3) + (25/4,0) $)},rotate=-135]
\node at (0,0) {1};

\node at (1,0) {0};
\node at (1/2,1/2) [rectangle,draw] {35};
\node at (0,1) {0};

\node at (2,0) {0};
\node at (3/2,1/2) [rectangle,draw] {0};
\node at (1,1) {149};
\node at (1/2,3/2) [rectangle,draw] {0};
\node at (0,2) {0};

\node at (3,0) {\ul1};
\node at (5/2,1/2) [rectangle,draw] {1};
\node at (2,1) {\ul1};
\node at (3/2,3/2) [rectangle,draw] {232};
\node at (1,2) {\ul1};
\node at (1/2,5/2) [rectangle,draw] {1};
\node at (0,3) {\ul1,};

\node at (3,1) {0};
\node at (5/2,3/2) [rectangle,draw] {0};
\node at (2,2) {149};
\node at (3/2,5/2) [rectangle,draw] {0};
\node at (1,3) {0};

\node at (3,2) {0};
\node at (5/2,5/2) [rectangle,draw] {35};
\node at (2,3) {0};

\node at (3,3) {1};
\end{scope}
\end{scope}

\begin{scope}[xscale=2,yscale=0.8]
\begin{scope}[shift={(3,-3)},rotate=-135]
\node at (0,0) {1};

\node at (1,0) {0};
\node at (1/2,1/2) [rectangle,draw] {126};
\node at (0,1) {0};

\node at (2,0) {0};
\node at (3/2,1/2) [rectangle,draw] {0};
\node at (1,1) {976};
\node at (1/2,3/2) [rectangle,draw] {0};
\node at (0,2) {0};

\node at (3,0) {0};
\node at (5/2,1/2) [rectangle,draw] {0};
\node at (2,1) {0};
\node at (3/2,3/2) [rectangle,draw] {2826};
\node at (1,2) {0};
\node at (1/2,5/2) [rectangle,draw] {0};
\node at (0,3) {0};

\node at (4,0) {\ul1};
\node at (7/2,1/2) [rectangle,draw] {1};
\node at (3,1) {\ul1};
\node at (5/2,3/2) [rectangle,draw] {1};
\node at (2,2) {\ul1+3951};
\node at (3/2,5/2) [rectangle,draw] {1};
\node at (1,3) {\ul1};
\node at (1/2,7/2) [rectangle,draw] {1};
\node at (0,4) {\ul1.};

\node at (4,1) {0};
\node at (7/2,3/2) [rectangle,draw] {0};
\node at (3,2) {0};
\node at (5/2,5/2) [rectangle,draw] {2826};
\node at (2,3) {0};
\node at (3/2,7/2) [rectangle,draw] {0};
\node at (1,4) {0};

\node at (4,2) {0};
\node at (7/2,5/2) [rectangle,draw] {0};
\node at (3,3) {976};
\node at (5/2,7/2) [rectangle,draw] {0};
\node at (2,4) {0};

\node at (4,3) {0};
\node at (7/2,7/2) [rectangle,draw] {126};
\node at (3,4) {0};

\node at (4,4) {1};
\end{scope}
\end{scope}

\end{tikzpicture}
\caption{\footnotesize The Hodge diamonds for $W^\vee = x_0^2 + x_1^{2n} + \cdots + x_n^{2n}$ and $H^\vee = \SL_W$, $n = 1,\ldots,5$.}
\label{fig:total-Hodge-mirror-fermat}
\end{figure}

Let us now turn to the dual pair: 
$W^\vee = W =x_0^2+x_1^{2n}+\dots+x_n^{2n}$ paired with $H^\vee = \SL_W$. Figure \ref{fig:total-Hodge-mirror-fermat} shows the Hodge diamonds for $H^{p,q}(\Sigma_{W,\SL_W};\CC)$ and that of 
$H^{p,q}_\sigma(W,\SL_W)$ with a shift of $(1/2,1/2)$. Again, within 
the Hodge diamond, we have underlined the ranks of the $\sigma$-anti-invariant subspaces.
Here, $\ol \SL_W$ is isomorphic to $(\ZZ/2n\ZZ)^{n-1}$ and a basis is given for instance by the elements $\frac 1{2n} (0,1,\ldots,0,2n-1), \ldots, \frac 1{2n} (0,0,\ldots,1,2n-1)$. $\SL_W$ is generated by $\ol \SL_W$ and $j_W$.

Comparing Figure~\ref{fig:total-Hodge-mirror-fermat}
with Figure~\ref{fig:total-Hodge-fermat}, we see 
that the indices within square boxes match after a right angle rotation: the $(\frac 12+\ZZ, \frac 12+\ZZ)$-graded part is given by the $(\frac 12+\ZZ, \frac 12+\ZZ)$-graded part of the mirror Hodge diamond, rotated by a right angle. 
When we look at the indices not inscribed in square boxes, we see that the underlined number match all non-underlined numbers: this is part $(i)$ of 
Theorem \ref{thm:semimirror}.
\end{exa}

\subsection{K3 surfaces with anti-symplectic involutions}\label{subsection:ABS}
A pair $(\Sigma,\sigma)$, formed by a K3 surface $\Sigma$ and 
an anti-symplectic involution
$\sigma\colon \Sigma\rightarrow \Sigma$, may be regarded as 
a lattice-polarized K3 surface: the polarization is 
given by the
 $\sigma$-invariant lattice $M = H^2(S,\ZZ)^\sigma$
within $\Lambda=H^2(S,\ZZ)$, which is 
equipped with a lattice structure isomorphic to 
$U^{\oplus 8}\oplus E_8(-1)^{\oplus 2}$ via the cup product taking values in $H^4(\Sigma;\ZZ)=\ZZ$.

Nikulin \cite{Ni1} showed that the lattices obtained in this
way are $2$-elementary: 
their discriminant group $\Hom(M,\ZZ)/M$ is isomorphic to $(\ZZ/2)^a$ 
for a some $a$. 
Two-elementary lattices are classified up to isometry by three invariants:
the rank of the lattice $r$, 
the rank $a$ of $\Hom(M,\ZZ)/M$ over $\ZZ/2$, and 
$\delta\in \{0,1\}$, vanishing if and only if $x^2 \in \ZZ$ for all $x \in \Hom(M,\ZZ)/M$.
All the possible 75 triples $(r,a,\delta)$ 
of the lattices $M$ arising from K3 surfaces with anti-symplectic involution
are pictured below. A dot, respectively a circle, in position $(r,a)$ 
indicates the existence of a K3 with involution 
whose invariants are $(r,a,1)$, respectively $(r,a,0)$.
 \begin{figure}[!ht] 
\begin{tikzpicture}[scale=.25]
\filldraw [black] 
(1,1) circle (5pt) 
(2,0) node[draw,circle,scale=0.6]{}
(2,2) circle (5pt) 
(2,2) node[draw,circle,scale=0.6]{}
 (3,1) circle (5pt) 
 (3,3) circle (5pt)
 (4,2) circle (5pt)
(4,4) circle (5pt)
(5,3) circle (5pt)
(5,5) circle (5pt)
(6,4) circle (5pt)node[draw,circle,scale=0.6]{} 
(6,2) node[draw,circle,scale=0.6]{} 
(6,6) circle (5pt) 
(7,3) circle (5pt) 
(7,5) circle (5pt)
(7,7) circle (5pt) 
(8,2) circle (5pt)
(8,4) circle (5pt)
(8,6) circle (5pt) 
(8,8) circle (5pt)
(9,1) circle (5pt) 
(9,3) circle (5pt)
(9,5) circle (5pt)
(9,7) circle (5pt)
(9,9) circle (5pt)
(10,0) node[draw,circle,scale=0.6]{}
(10,2) circle (5pt)node[draw,circle,scale=0.6]{} 
(10,4) circle (5pt)node[draw,circle,scale=0.6]{}
(10,6) circle (5pt)node[draw,circle,scale=0.6]{}
(10,8) circle (5pt)node[draw,circle,scale=0.6]{}
(10,10) circle (5pt)node[draw,circle,scale=0.6]{}
(11,1) circle (5pt) 
(11,3) circle (5pt)
(11,5) circle (5pt)
(11,7) circle (5pt)
(11,9) circle (5pt)
(11,11) circle (5pt) 
(12,2) circle (5pt)
(12,4) circle (5pt)
(12,6) circle (5pt) 
(12,8) circle (5pt)
(12,10) circle (5pt) 
(13,3) circle (5pt) 
(13,5) circle (5pt)
(13,7) circle (5pt) 
(13,9) circle (5pt) 
(14,2) node[draw,circle,scale=0.6]{}
(14,4) circle (5pt)node[draw,circle,scale=0.6]{}
(14,6) circle (5pt)node[draw,circle,scale=0.6]{}
(14,8) circle (5pt) 
(15,3) circle (5pt)
(15,5) circle (5pt)
(15,7) circle (5pt)
(16,2) circle (5pt)
(16,4) circle (5pt)
(16,6) circle (5pt) 
(17,1) circle (5pt) 
(17,3) circle (5pt)
(17,5) circle (5pt)
(18,0) node[draw,circle,scale=0.6]{}
(18,2) circle (5pt)node[draw,circle,scale=0.6]{}
(18,4) circle (5pt) node[draw,circle,scale=0.6]{}
(19,1) circle (5pt)
(19,3) circle (5pt)
(20,2) circle (5pt)
 ;
 

\draw[->] (10,12) -- coordinate (N' axis mid) (-1,1);
\draw[->] (10,12) -- coordinate (N axis mid) (21,1);
\draw[->] (0,0) -- coordinate (x axis mid) (22,0);
 \draw[->] (0,0) -- coordinate (y axis mid)(0,13);
 \foreach \x in {10,20}
 \draw [xshift=0cm](\x cm,0pt) -- (\x cm,-3pt)
 node[anchor=north] {$\x$};
 \foreach \y in {5,10}
 \draw (1pt,\y cm) -- (-3pt,\y cm) node[anchor=east] {$\y$};
 \node[below=-0.3cm, right=-0.5cm] at (N' axis mid) {$g$};
 \node[below=-0.3cm, right=0cm] at (N axis mid) {$N$};
 \node[below=0.3cm, right=2.6cm] at (x axis mid) {$r$};
 \node[left=0.3cm, below=-1.6cm] at (y axis mid) {$a$};
 \end{tikzpicture} 
 \end{figure}
 
\noindent The twelve cases 
satisfying $r + a = 22$ 
or $(r,a,\delta)=(14, 6, 0)$ are special.
They are precisely the cases we need to take off for 
the figure to possess a symmetry with 
respect to the vertical axis $r=10$. The explanation is 
mirror symmetry of lattice polarized
K3 surfaces.
Voisin \cite{Vo} proved that the 2-elementary lattices $M = H^2(S,\ZZ)^\sigma$, 
which are not among the twelve special cases 
($r + a = 22$ 
or $(r,a,\delta)=(14, 6, 0)$),
are exactly those possessing a
perpendicular lattice $M^\perp$ within 
$\Lambda$
satisfying 
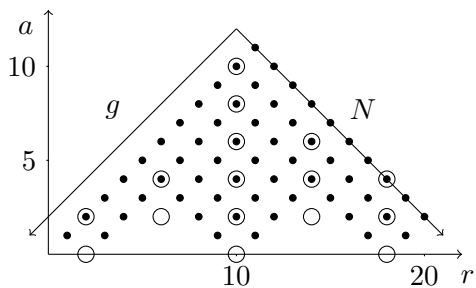
\begin{equation}\label{eq:precondition}
M^\perp \cong U \oplus M^\vee.\end{equation}
We refer to $M^\vee$ as the mirror lattice, and we notice that 
 $(M^\vee)^\perp$ is isomorphic to $U\oplus M$; hence $(M^\vee)^\vee=M$.
 For such lattices, the mirror lattice $M^\vee$ 
 has invariants $(20-r,a,\delta)$. This explains the symmetry 
 appearing within the picture given above.

Dolgachev constructs a coarse 
moduli space $\mathcal{K}_M$ attached to any such lattice and classifying 
$M$-polarized K3 surfaces, \emph{i.e.} 
pairs $(S,j)$, 
where $S$ is a K3 surface and $j\colon M \hookrightarrow \Pic(S)$ is a primitive lattice embedding (this holds for an
even non-degenerate 
lattice $M$ of signature $(1,\rho-1)$, $1 \leq \rho \leq 19$ 
with 
a primitive embedding $M \hookrightarrow \Lambda$).
Two lattice-polarized 
K3 surfaces with an anti-symplectic involution
form a mirror pair if they are represented by two points lying in the two mirror spaces 
$\mathcal{K}_M$ and $\mathcal{K}_{M^\vee}$.

In the statement below, we summarize the connection between the lattice invariants $r$ and $a$ and the topological invariants of the K3 surface $\Sigma$ and the 
involution $\sigma$. We recall (see \cite{AST}) that the 
rank $r$ is related to the Euler characteristic
of the $\sigma$-fixed locus $\Sigma_\sigma=C$ as follows
\[ \chi(C)=2r-20 \]
(the right hand side is the trace of $\sigma$ on $H^{1,1}$, Lefschetz fixed point theorem). On the other hand, 
by the Smith exact sequence,
the rank $2a$ is the difference between the 
dimension of the cohomology of $\Sigma$ and 
of 
$C$
\[\dim H^*(\Sigma;\CC)-2a=\dim H^*(C;\CC),\]
unless $C=\varnothing$ where the above formula holds with $4$ on the right hand side.
This yields the following relations.
\begin{pro}[\cite{Ni1}, Theorem 4.2.2]
Let $\Sigma$ be a
K3 
surface with an anti-symplectic involution $\sigma$.
The $\sigma$-fixed locus $C$ is a 
disjoint union of $k$ smooth curves $C_1,\dots,C_k$ whose total genus equals $g=\sum_i g(C_i)$. 
Let $r$ and $a$ be 
the rank of the lattice $M=H^2(\Sigma;\ZZ)_\sigma$, 
and of $\Hom(M,\ZZ)/M$.
We have 
\[
k=\frac{r-a}2+1, \qquad g=-\frac{r+a}2+11.
\]
except when $(r,a,\delta)=(10,10,0)$ where the fixed locus is empty, \emph{i.e.} 
$k=0$ and $g=0$.
Except from the case $(r,a,\delta)=(10,8,0)$,
where $C$ is the union of two elliptic curves,
the fixed locus contains at most 
one single component of genus $g>0$
and is topologically determined by $g$ and $k$: we
have
$C\cong C_1\sqcup \bigsqcup_{i=2}^{k} \PP^1$ (with 
$g(C_1)=g$).
\qed
\end{pro}
In view of the above lemma, mirror duality can be 
regarded as a symmetry along the axis $k=g$ interchanging 
$k$ with $g$.
Indeed, by Proposition \ref{pro:crthm}, the invariants $k$ and $g$ equal respectively $h_\sigma^{0,0}(\frac12)$ and $h_\sigma^{1,0}(\frac12)$. 
Artebani, Boissière and Sarti \cite{ABSBHCR} 
compute the corresponding invariants $(r,a,\delta)$ in all possible cases of Berglund--H\"ubsch duality.
Out of the 75 Nikulin's possible triples $(r,a,\delta)$ only 29 possible triples $(r,a,\delta)$ arise via Berglund--H\"ubsch duality.
Neither the twelve special triples without mirror, nor the single case with empty $\sigma$-fixed locus, nor the single case with $\sigma$-fixed locus given by two elliptic curves ever occur among these $29$ cases. 
Furthermore, if the invariant attached to $(W,G)$ equals $(r,a,\delta)$, then the invariant of the Berglund--Hübsch mirror
$(W,G^\vee)$ equals $(20-r,a, \delta)$. This proves the compatibility of 
Berglund--Hübsch construction with the lattice mirror symmetry of polarized K3 surfaces. 
The computation of \cite{ABSBHCR} is in several cases spectacular; see for instance Example \ref{exa:inhomog} below. 
The proof of the compatibility between the two mirror constructions relies on a case-by-case study and is often based on a computer calculation. Clearly, not all computations are explicit in the literature. 

The present paper remedies this. We point out how Theorem \ref{thm:semimirror} and Proposition \ref{pro:crthm}
yield a conceptual 
understanding of the relations 
$r^\vee=20-r,a^\vee=a$ as a consequence of the fact that $h_\sigma^{0,0}(\frac12)$ and $h_\sigma^{1,0}(\frac12)$ are exchanged by mirror duality. We obtain the following 
corollary.
\begin{cor}\label{cor:semimirror}
In the same conditions as in Thm.~\ref{thm:semimirror} we set $n=3$
so that $\Sigma_{W,H}$ and 
$\Sigma_{W^\vee,H^\vee}$ are $2$-dimensional stacks
and we write $\wt \Sigma$ and 
$\wt \Sigma^\vee$ for the 
K3 surfaces arising from the minimal resolutions
of their coarse spaces.
We denote by 
$\sigma$ their anti-symplectic involutions
and by $C$ and $C^\vee$ their respective fixed loci,
which are disjoint unions of $k$ and $k^\vee$ smooth curves whose total genus equals $g$ and $g^\vee$.
Then we have 
\[H^{p,q}(\wt \Sigma;\CC)^\pm\cong H^{2-p,q}(\wt \Sigma^\vee;\CC)^\mp,\qquad H^{p,q}(C;\CC)\cong H^{1-p,q}(C^\vee;\CC).\]
In other words we have 
\[\mathrm{rk}(H^2(\wt\Sigma;\ZZ)_\sigma)=20-\mathrm{rk}(H^2(\wt\Sigma^\vee;\ZZ)_\sigma),\qquad
g=k^\vee,\qquad 
k=g^\vee.\]
\end{cor}

We illustrate the result with an example.

\begin{exa}\label{exa:inhomog}
Consider the degree-$18$ polynomial 
\[ W = x_0^2 + f(x_1,x_2,x_3,x_4)=x_0^2+x_1^4x_3+x_3^7x_1+x_2^6,\] 
where the variables have weights $(9,4,3,2)$. 
Consider the group $H=J_W$, which coincides with $\SL_W$ in this case. 
The action by $\sigma$ clearly fixes the curve $\{x_1^4x_3+x_3^7x_1+x_2^6=0\}$ within the linear subspace $\{x_0=0\}=\PP(4,3,2)\subset \PP(9,4,3,2)$. 
It is crucial, however, to notice that $\sigma$ fixes also $\{x_2=0\}$; indeed if we compose $\sigma$ with
with the weighted $(9,4,3,2)$-action of $\la=-1$ we get a diagonal action 
fixing every variable except $x_2$, whose sign is changed. 
As a result, the fixed locus is larger than $Z_{\PP(w_1,w_2,w_3)}(f)$. In this example, one can show that it 
is connected but not irreducible, and not even smooth: the curve $C=\{x_0=0, x_1^4x_3+x_3^7x_1+x_2^6=0\}$ 
and the curve $R=\{x_2=0, x_0^2+x_1^4x_3+x_3^7x_1=0\}$ 
intersect at 5 points. 
In the light of Proposition~\ref{pro:crthm} and 
the argument of its proof we are looking at a twisted curve lying as a closed 
substack within $\Sigma_W=\Sigma_{W,H}$; notice 
that it has stabilizers of even order at the nodes.

We now compute the $\sigma$-orbifold cohomology of $\Sigma_W$, which is the cohomology of the $\sigma$-inertia stack $\mf I_{\Sigma_W}^\sigma$. By Proposition \ref{pro:crthm}, this coincides with the cohomology of the fixed space of the resolution. We apply \eqref{eq:sigmaorbifoldmadeexplicit}. 
More precisely, there are four values of $\la$ for which the hypersurface 
$Z(W_{\sigma\la})$ in $\CC^4_{\sigma\la}\setminus \pmb 0$ is non-empty. These are the fourth roots of unity. 

For $\la=1$, we examine the hypersurface defined by 
the restriction of $W$ to the linear subspace defined 
by $x_1,x_2,$ and $x_3$. This is the curve 
$\{x_0=0, x_1^4x_3+x_3^7x_1+x_2^6=0\}$ fixed by~$\sigma$. 
The standard genus formula within weighted projective spaces 
or the computation of primitive cohomology via the Milnor ring show that this curve 
has genus $3$. The contribution to $h^{*,*}_\sigma(\Sigma_{W};\CC)(\frac12)$ is 
precisely $1$ in bidegrees $(0,0)$ and $(1,1)$ and
$3$ in bidegrees $(1,0)$ and $(0,1)$ (note that the age is $\frac12$).

For $\la=-1$, we examine the hypersurface 
$\{W_{\sigma\la}=0\}$ modulo $\sigma$ defined by the 
restriction of $W$ to the linear subspace fixed by
$\sigma\la$ which acts by multiplication by 
$1,1,-1,$ and $1$ on $x_0$, $x_1$, $x_2$, and 
$x_3$. This is the curve 
$\{x_2=0, x_0^2+x_1^4x_3+x_3^7x_1=0\}$, whose coarse space 
is a (rational) double cover
of $\PP(4,2)$.
The contribution to $h^{*,*}_\sigma(\Sigma_{W};\CC)(\frac12)$ is $1$ 
in both bidegrees $(0,0)$ and $(1,1)$. 

For $\la=\cxi$, we notice that $\sigma \la$ acts as $\frac14(3,0,3,2)$
and that $W$ vanishes identically on the fixed space. 
The age shift is $\age (\sigma\la)-\frac12=\frac32$ by 
a straightforward application of Remark \ref{rem:agecorrection}. 
This is the age of 
the vector bundle tangent to ${[\PP(\textstyle{\frac{d}2},w_1,\dots,w_n)/\sigma]}$ minus the 
age of the line bundle normal to $[\Sigma_W/\sigma]$. The latter 
is linearized by a character of
weight $\deg W =18$; since $\langle18/4\rangle=1/2$ this yields the above correction $-1/2$ (see Remark \ref{rem:agerounddown}). The contribution to $h^{*,*}_\sigma(\frac12)$ 
is $1$ 
in bidegree $(1,1)$. 

The analysis of the case $\la=-\cxi$ is completely analogous, 
$\sigma \la$ acts as $\frac14(1,0,1,2)$ and that $x_0^2+W$ vanishes identically on the fixed space. The age is $1-\frac12=\frac12$ (by the same argument as above).
The contribution to $h^{*,*}_\sigma(\frac12)$ 
is $1$ 
in bidegree $(0,0)$.

In Figure~\ref{fig:total-Hodge-inhomog}, we represent the Hodge diamond of $H^{*,*}_{\CR}(\Sigma_W;\CC)$, which is 
the usual K3 surface Hodge diamond, and --- within it --- that of $H_\sigma^{*,*}$.
As above, we record the ranks of the 
$\sigma$-anti-invariant subspaces by underlining 
all the corresponding entries in the Hodge diamond.

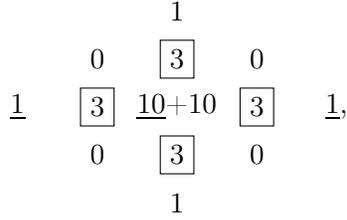
\begin{figure}
\begin{tikzpicture}[xscale=1.5,yscale=0.9]
\begin{scope}[rotate=-135]
\node at (0,0) {1};

\node at (1,0) {0};
\node at (1/2,1/2) [rectangle,draw] {3};
\node at (0,1) {0};

\node at (2,0) {\ul 1};
\node at (3/2,1/2) [rectangle,draw] {3};
\node at (1,1) {\ul{10}+10};
\node at (1/2,3/2) [rectangle,draw] {3};
\node at (0,2) {\ul 1,};

\node at (2,1) {0};
\node at (3/2,3/2) [rectangle,draw] {3};
\node at (1,2) {0};

\node at (2,2) { 1};
\end{scope}
\end{tikzpicture}
\caption{\footnotesize Total Hodge diamond in Example~\ref{exa:inhomog}.}
\label{fig:total-Hodge-inhomog}
\end{figure}

Regarding the mirror side, notice that the polynomial $W^\vee$ is equal to $W$ and that $\SL_W$ coincides with $(j_W^2)$. Therefore Theorem~\ref{thm:semimirror} predicts that the Hodge diamond appearing in Figure~\ref{fig:total-Hodge-inhomog} is stable with respect to right angle rotations.

This symmetry is the result of the fact that the $\sigma$-invariant part and anti-invariant part coincide up to a right-angle rotation, and of the fact that the $(\frac 12 + \ZZ, \frac 12 +\ZZ)$-graded part is itself symmetric.

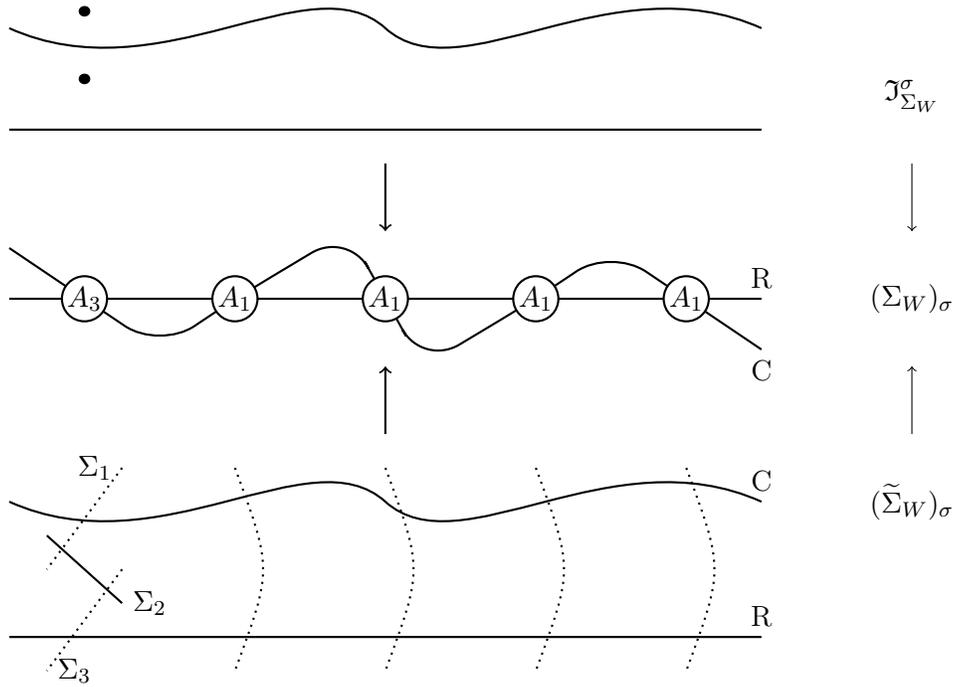
\begin{figure}[h]
\begin{tikzpicture}
[singularity/.style={circle,draw,fill=white,thick,
inner sep=0pt,minimum size=6mm},yscale=0.9]

\filldraw (-4,4.25) circle (2pt);
\filldraw (-4,3.25) circle (2pt);

\draw[thick] (-5,4) .. controls (-3,3) and (-1,5) .. (0,4)
 .. controls (1,3) and (3,5) .. (5,4);

\draw[thick] (-5,2.5) -- (5,2.5);

\node at (7,3) {$\mf I_{\Sigma_W}^\sigma$};
\node at (7,0) {$(\Sigma_W)_\sigma$};
\node at (7,-3) {$(\wt \Sigma_W)_\sigma$};

\draw[->] (7,2) -- (7,1);
\draw[->] (7,-2) -- (7,-1); 

\draw[thick,->] (0,2) -- (0,1); 
\draw[thick,->] (0,-2) -- (0,-1); 

\draw[thick] (-5,0) -- (5,0) node[above] {R};
\draw[thick,rounded corners=16pt] (-5,0.75) -- (-4,0) node [singularity] {$A_3$} -- (-3,-0.75) -- (-2,0) node [singularity] {$A_1$} -- (-0.5,1) -- (0,0) node [singularity] {$A_1$} -- (0.5,-1) -- (2,0) node [singularity] {$A_1$} -- (3,0.75) -- (4,0) node [singularity] {$A_1$} -- (5,-0.75) node[below] {C};

\draw[thick] (-5,-3) .. controls (-3,-4) and (-1,-2) .. (0,-3)
 .. controls (1,-4) and (3,-2) .. (5,-3) node[above] {C};

\draw[thick,dotted] (-2,-2.5) .. controls (-1.5,-4) .. (-2,-5.5);
\draw[thick,dotted] ( 0,-2.5) .. controls ( 0.5,-4) .. ( 0,-5.5);
\draw[thick,dotted] ( 2,-2.5) .. controls ( 2.5,-4) .. ( 2,-5.5);
\draw[thick,dotted] ( 4,-2.5) .. controls ( 4.5,-4) .. ( 4,-5.5);

\draw[thick,dotted] (-3.5,-2.5) node[left] {$\Sigma_1$} -- (-4.5,-4);
\draw[thick] (-4.5,-3.5) -- (-3.5,-4.5) node[right] {$\Sigma_2$};
\draw[thick,dotted] (-3.5,-4) -- (-4.5,-5.5) node[right] {$\Sigma_3$};

\draw[thick] (-5,-5) -- (5,-5) node[above] {R} ;

\end{tikzpicture}
\caption{\footnotesize The $\sigma$-fixed twisted curve $(\Sigma_W)_\sigma$, the fixed curve within the K3 resolution $(\wt \Sigma_W)_\sigma$ and the 
$\sigma$-inertia stack $\mf I_{\Sigma_W}^\sigma$ defined by $W =x_0^2 +x_1^4x_3 + x_3^7x_1 + x_2^6$.}
\label{fig:example}
\end{figure}

In \cite[Exa.~5.1]{ABSBHCR}, the authors resolve the coarse space of $\Sigma_W$ and 
study the fixed locus of the involution induced by $\sigma$
on the resolution $\wt \Sigma_W\to \Sigma_W$. The fixed locus consists of 
$3$ connected components: a genus-$3$ curve and two projective lines. 
As a consequence of Proposition \ref{pro:crthm}, the Hodge diamond of $H^*(\wt \Sigma_W;\CC)$
matches that of $H_{\sigma}^{*}(\Sigma_{W})(\frac12)$ appearing within boxes in Figure~\ref{fig:total-Hodge-inhomog}: $h^{0,0}=3, h^{1,0}=3, h^{0,1}=3, h^{1,1}=3$.

We illustrate in Figure \ref{fig:example} all the different pictures involved 
here:\begin{enumerate}
\item the $\sigma$-inertia stack $\mf I_{\Sigma_W}^\sigma$ which records separately the fixed locus of $\sigma \lambda$ for each $\lambda$ a fourth root of unity. 
\item the $\sigma$-fixed locus $(\Sigma_W)_\sigma$, \emph{i.e.} the twisted curve $C\cup R$ described above;
\item the smooth curve $(\wt\Sigma_W)_\sigma$ fixed within the K3 surface $\wt \Sigma_W$ in the following picture.
\end{enumerate}
Indeed, the resolution of the three simple singularities occurring at the nodes of 
the twisted curves yields chains of curves 
of the same length as their singularity 
index. It is now easy to detect the fixed locus by 
knowing that the genus-$3$ curve $C$ and the rational curve $R$ are fixed and the chains contain alternatively 
$\sigma$-fixed subcurves and 
moving subcurves, where $\sigma$ is given by $\sigma \colon \PP_z^1\to \PP_z^1; z\mapsto -z$.
These moving rational curves are those 
which share a point with $C$ or $R$. Only the chain over the $A_3$-singularity yields a new fixed component $\Sigma_2$ (see Figure~\ref{fig:example}). The fixed locus is $C\sqcup R\sqcup \Sigma_2$, and its Hodge diamond matches the diamond given above:
$h^{0,0}=3, h^{1,0}=3, h^{0,1}=3, h^{1,1}=3$.
\end{exa}

\subsection{Borcea--Voisin Mirror Symmetry
in any dimension} \label{sect:generalized-BV}
The classical Borcea--Voisin construction involves 
an elliptic curve $E$ with its hyperelliptic involution
$\sigma_1$ and a K3 surface $K$ with anti-symplectic involution $\sigma_2$, and a crepant resolution of the quotient $E \times K/(\sigma_1 \times \sigma_2).$ From this setup
we obtain some of the earliest examples of mirror symmetry for Calabi--Yau threefolds. 
Consider 
$E \times K/(\sigma_1 \times \sigma_2)$ and $E \times K^\vee/(\sigma_1 \times \sigma^\vee_2)$
for any K3 surface $K^\vee$ with anti-symplectic involution $\sigma_2^\vee$ 
mirror to $(K,\sigma_2)$.
 Because the two quotients are three-dimensional 
and Gorenstein, 
crepant resolutions $\tilde \Sigma$ and $\tilde \Sigma^\vee$ 
exist and yield 
two mirror Calabi--Yau three-folds $\tilde\Sigma$ and $\tilde \Sigma^\vee$ 
satisfying 
\begin{equation}\label{eq:BV3fold}H^{p,q}(\tilde \Sigma;\CC)\cong H^{3-p,q}(\tilde \Sigma^\vee;\CC).\end{equation}

The point of view of this paper is that the mirror duality above, 
suitably stated, only relies on the properties proven in Theorem 
\ref{thm:semimirror}. For example, any elliptic curve 
alongside with its hyperelliptic involution trivially satisfies 
conditions (i) and (ii) of Theorem \ref{thm:semimirror} 
($h^{i,j}(E;\CC)=1$ for $i,j\in {0,1}$ and 
$h^{i,j}(E_\sigma;\CC)\neq 0$ only if $i,j=0$). Therefore any choice of elliptic curves on each side of the 
above duality leads to
Calabi--Yau three-folds satisfying \eqref{eq:BV3fold}. 
By considering the framework of Theorem \ref{thm:semimirror} we get a natural corollary 
generalising the above statement. 
Let $\Sigma_1=\Sigma_{W_1,H_1}$ and $\Sigma_2=\Sigma_{W_1,H_2}$ be 
Calabi--Yau orbifolds attached to two invertible CY polynomials with involutions $W_1=(x_0^1)^2+f_1$ and $W_2:=(x_0^2)^2+f_2$ and two groups $H_1$ and $H_2$ fitting in 
$j_{W_1}\in H_1 \subset \SL_{W_1}$ and $j_{W_2}\in H_2\subset\SL_{W_2}$. 
Then, on both sides we have involutions 
\[ \sigma_1\colon \Sigma_1\to \Sigma_1\qquad\text{and}\qquad \sigma_2\colon \Sigma_2\to \Sigma_2\]
and, via Berglund--Hübsch, mirror partners 
$\Sigma^\vee_1=\Sigma_{W^\vee_1,H^\vee_1}$ and $\Sigma^\vee_2=\Sigma_{W^\vee_2,H^\vee_2}$, with involutions 
\[ \sigma^\vee_1\colon \Sigma^\vee_1\to \Sigma^\vee_1\qquad\text{and}\qquad \sigma^\vee_2\colon \Sigma^\vee_2\to \Sigma^\vee_2.\]
Theorem \ref{thm:semimirror} applies and we now show how it leads us to 
the cohomological mirror duality 
\begin{equation}\label{eq:CRBVmirror}
	H_{\CR}^{p,q}\left(\Sigma;\CC\right)
\cong H_{\CR}^{d-p,q} 
(\Sigma^\vee;\CC).
\end{equation}
between the $d$-dimensional orbifolds
$\Sigma=[\Sigma_1\times \Sigma_2/(\sigma_1,\sigma_2)]$ \text{and} 
$\Sigma^\vee=[\Sigma^\vee_1\times \Sigma^\vee_2/(\sigma^\vee_1,\sigma^\vee_2)].$
If crepant resolutions $\tilde\Sigma$ and $\tilde \Sigma^\vee$ 
exist, Theorem \ref{thm:crthm}
leads to 
$H^{p,q}(\tilde \Sigma;\CC)\cong H^{d-p,q}(\tilde \Sigma^\vee;\CC).$

We prove the above theorem in a more general form admitting 
any number $n$ of factors. The involution $(\sigma_1,\sigma_2)$ is 
replaced by a class of subgroups $G$ of $(\ZZ/2)^{n}=\prod_i(\sigma_i)$ 
specialising for $n=2$ to the case of the 
order-$2$ subgroup spanned by $(\sigma_1,\sigma_2)<
(\ZZ/2)\times(\ZZ/2)$. (For each even $n$ the construction includes 
the 
order-$2$ subgroup spanned by $(\sigma_1,\dots,\sigma_{2n})<
(\ZZ/2)^{2n}$ which we refer to in the introduction. 

Each symmetry of $(\ZZ/2)^{n}$ is of the following 
form for $I\in [n]$
\[\sigma_I=(\sigma_1^{a_1},\dots,\sigma_n^{a_n})\qquad 
\text{with $a_i=0$ if and only if $i \not \in I$}.\]
Then $G$ is called \emph{admissible} if any two elements $\sigma_{I}$, $\sigma_J$ satisfy the condition
$|I\setminus J|\in 2\NN$. 
Note that, since $G$ is a group, we have $\sigma_\varnothing \in G$ and therefore 
$|I|\in 2\ZZ$ for all $\sigma_I$; if we regard the elements of 
$(\ZZ/2)^n$ as an $n$-dimensional representation in $\GL(n;\CC)$, 
this means in particular that $G$ lies in $\SL(n;\CC)$. 
Furthermore, $\sigma_I\sigma_J=\sigma_{I\Delta J}$ for 
$I\Delta J=I\setminus J \sqcup J\setminus I$. Therefore 
the condition 
$|I\setminus J|\in 2\NN$ is symmetric: it is equivalent to 
$|J\setminus I|\in 2\NN$ because $|I\Delta J|$ is even. 

%
%
%
\begin{thm}\label{thm:mirror} 
For $i=1,\dots n$, let 
$(W_i,H_i)$ be a pair of a Calabi--Yau invertible polynomial of the 
form $W_i=(x_0^i)^2+f_i(x_1^i,\dots, x_{m_i}^i)$ with $J_W\in H_i\subseteq \SL_{W_i}$. 
Let $G$ be an admissible subgroup of $(\ZZ/2)^n$.
For $m=\sum_i m_i$, 
set the $(m-n)$-dimensional Calabi--Yau orbifolds
\[
\Sigma=
\left[\prod_i \Sigma_{W_i,H_i}/G\right]\qquad
\text{and}
\qquad
\Sigma^\vee=
\left[\prod_i \Sigma_{W^\vee_i,H^\vee_i}/G\right]\]
Then, we have 
\[
	H_{\CR}^{p,q}\left(\Sigma;\CC\right)
\cong H_{\CR}^{m-n-p,q} 
(\Sigma^\vee;\CC).
\]
 \end{thm}
By Theorem \ref{thm:crthm}, as an immediate consequence, we have the following statement.
\begin{cor}\label{cor:BVres}
If $\Sigma$ and $\Sigma^\vee$ admit a crepant resolution $\wt \Sigma$ and $\wt \Sigma^\vee$, then $H^{p,q}(\wt\Sigma;\CC)\cong 
H^{n+m-2-p,q}(\wt \Sigma^\vee;\CC)$.
\end{cor}

\noindent \textit{Proof of Theorem \ref{thm:mirror}.}
The stack $\Sigma$ is the quotient stack 
$[U/H]$, where 
$U$ is the locus within $\CC^{m+n}-Z$ where the polynomials $W_i$ vanish. Here 
\[ Z=\{(z^1_0,\dots,z^1_{m_1},\dots,z^n_{0},\dots, z^n_{m_n}) \mid z^i_0 = \cdots =z^i_{m_i}=0 \text{ for some } i\}. \] 
To define $H$, embed $G$ as a subgroup of $\Aut(W):=\Aut(W_1) \times \cdots \times \Aut(W_n)$, and consider the map $\phi: (\GG_m)^n \to \Aut(W)$ defined by 
\[(\lambda_1,\dots,\lambda_n) \mapsto (\lambda_1^{w^1_0},\dots,\lambda_1^{w^1_{m_1}},\dots,\lambda_n^{w^n_0},\dots,\lambda_n^{w^n_{m_n}})\]
where the $w^i_j$ are the weights of the $x^i_j$. Then $H$ is the the group generated by $G$, $H_1 \times \cdots \times H_n$, and $\phi((\GG_m)^n).$ 

As discussed in Section \ref{sect:thestackweconsider}, we can use the quotient stack presentation of the inertia stack to determine Chen--Ruan's orbifold cohomology. The cohomology is a direct sum over each 
element in $H$ with non-empty fixed locus in $U$. We claim that there are only finitely many such elements. It suffices to show that there are only finitely many elements of $(\GG_m)^n$ such that $\phi((\lambda_i)_{i=1}^n) g$ has a non-empty fixed locus for each $g$ in the subgroup generated by $H_1 \times \dots \times H_n$ and $G$. Writing $g$ as $(g^1_0,\dots,g^n_{m_n})$, we see that in order for $\phi((\lambda_i)_{i=1}^n) g$ to have a non-empty fixed locus, for every $i$ there must be at least one $j$ such that $\lambda_i^{w_{ij}} g^i_j=1$. This is because $U$ has empty intersection with $Z$ by definition. As only finitely many $\lambda_i$ can satisfy this for each~$g$, this shows the claim.

For every $\ga \in H$ with non-empty fixed locus, we write 
$(\al_1,\dots,\al_n)$, separating the coordinates from each polynomial. Then the contribution to Chen--Ruan's cohomology 
in bidegree $(p,q)$ is a cohomology group 
\begin{equation*}
H^{h,k}\left(\left(\left(Z(W_1)_{\al_1}\times \cdots \times Z(W_n)_{\al_n}\right)/(H_1\times \cdots \times H_n)\right)/G \right),
\end{equation*}
where $(h,k)\in \ZZ\times \ZZ$ satisfies 
\[
	(h,k)+(\mathrm{age}(\ga),\mathrm{age}(\ga))=(p,q).
\]
Here $Z(W_i)_{\al_i}$ is quotient of the vanishing of $(W_i)_{\al_i}$ in $(\CC^{m_i+1})^{\alpha_i}$ by $\GG_m$. Notice that 
$\left(Z(W_1)_{\al_1}\times \cdots \times Z(W_n)_{\al_n}\right)/(H_1\times \cdots \times H_n)$
equals the product of $n$
projective varieties with finite group quotient singularities 
\[
X_1 \times \cdots \times X_n;
\]
so, 
the $(h,k)$-graded cohomology decomposes as 
\[
	\bigoplus_{\substack{\sum_{i=1}^n h_i=h\\ \sum_{i=1}^n k_i=k}}\left(\bigotimes_{i=1}^n H^{h_i,k_i}(X_i;\CC) \right)^{G}.
\]
Suppose $\ga$ was in the coset $g (H_1 \times \cdots \times H_n) \phi((\GG_m)^n),$ where $g=\sigma_I$ for $I \subset \{1,\dots,n\}$. Then each choice of $h_i,k_i$ gives
\begin{equation} \label{Ginv} \bigotimes_{i \in I} H^{h_i,k_i}_\sigma (\Sigma_{W_i,H_i};\CC)^+ \otimes \left(\bigotimes_{i \in \overline{I}}H^{h_i,k_i}(\Sigma_{W_i,H_i};\CC)\right)^G
\end{equation}
with $(\ )^+$ and $(\ )^-$ denoting the involution-invariant and 
involution-anti-invariant subspaces, and $\overline{I}$ the complement of $I$. This can be further decomposed to a sum over $J \subset \overline{I}$, where the contribution from a given $J$ is
\[\bigotimes_{i \in I} H^{h_i,k_i}_\sigma (\Sigma_{W_i,H_i};\CC)^+ \otimes \left(\bigotimes_{i \in J} H^{h_i,k_i}(\Sigma_{W_i,H_i};\CC)^- \otimes \bigotimes_{i \in \overline{I} \setminus J} H^{h_i,k_i}(\Sigma_{W_i,H_i};\CC)^+\right)^G.\]
This is non-empty only if $J$ satisfies $|J \cap I'|\in 2\ZZ$ for all $I'$ such that $\sigma_{I'} \in G$. The contribution from such a $J$ is
\[\bigotimes_{i \in I} H^{h_i,k_i}_\sigma (\Sigma_{W_i,H_i};\CC)^+ \otimes \bigotimes_{i \in J} H^{h_i,k_i}(\Sigma_{W_i,H_i};\CC)^- \otimes \bigotimes_{i \in \overline{I} \setminus J} H^{h_i,k_i}(\Sigma_{W_i,H_i};\CC)^+.\]
Because $G$ is admissible, $|J \cap I'|\in 2\ZZ$ if and only if $|(\overline{I} \setminus J) \cap I'|\in 2\ZZ$, as $|\overline{I} \cap I'| \in 2\ZZ.$ Therefore, we could alternatively write equation \ref{Ginv} as a sum over $J$ satisfying the same conditions, but contributing
\[\bigotimes_{i \in I} H^{h_i,k_i}_\sigma (\Sigma_{W_i,H_i};\CC)^+ \otimes \bigotimes_{i \in J} H^{h_i,k_i}(\Sigma_{W_i,H_i};\CC)^+ \otimes \bigotimes_{i \in \overline{I} \setminus J} H^{h_i,k_i}(\Sigma_{W_i,H_i};\CC)^-.\]
Theorem \ref{thm:semimirror} says that this space is isomorphic to
\[\bigotimes_{i \in I} H^{m_i-h_i-1,k_i}_\sigma (\Sigma_{W^\vee_i,H^\vee_i};\CC)^+ \otimes \bigotimes_{i \in J} H^{m_i-h_i-1,k_i}(\Sigma_{W^\vee_i,H^\vee_i};\CC)^- \otimes \bigotimes_{i \in I \setminus J} H^{m_i-h_i-1,k_i}(\Sigma_{W^\vee_i,H^\vee_i};\CC)^+.\]
Applying the same argument to the mirror, we obtain that \eqref{Ginv} is isomorphic to
\[ \bigotimes_{i \in I} H^{m_i-h_i-1,k_i}_\sigma (\Sigma_{W_i^\vee,H_i^\vee};\CC)^+ \otimes \left(\bigotimes_{i \in \overline{I}}H^{m_i-h_i-1,k_i}(\Sigma_{W_i^\vee,H_i^\vee};\CC)\right)^G.\]
Summing over all choices of $h_i,k_i$ and $\gamma$ proves the theorem. 
\qed

\begin{rem} Part of the above proof is just a check of 
Künneth formula for Chen--Ruan cohomology, which can be found in 
\cite{Hepworth} in a more general setup. 
We provide an explicit treatment because the present situation 
requires a slightly 
more detailed analysis of 
invariant and anti-invariant cohomology. \end{rem}

\bibliographystyle{amsplain}
\bibliography{references}

\end{document}